\documentclass[12pt,draft]{amsart}

\makeatletter

\theoremstyle{plain}
\newtheorem{thm}{Theorem}[section]
\numberwithin{equation}{section} 
\numberwithin{figure}{section} 
\theoremstyle{plain}
\newtheorem*{thm*}{Theorem}
\theoremstyle{plain}
\newtheorem{cor}[thm]{Corollary} 
\theoremstyle{plain}
\newtheorem*{cor*}{Corollary}

\theoremstyle{plain}
\newtheorem{lem}[thm]{Lemma} 
\theoremstyle{plain}
\newtheorem{prop}[thm]{Proposition} 
\theoremstyle{definition}
\newtheorem{defn}[thm]{Definition}
\theoremstyle{remark}
\newtheorem{rem}[thm]{Remark}
\theoremstyle{remark}

\theoremstyle{remark}

\theoremstyle{remark}
\newtheorem{question}[thm]{Question}
\theoremstyle{definition}
\newtheorem{example}[thm]{Example}
\theoremstyle{remark}
  \newtheorem*{acknowledgement*}{Acknowledgement}

\theoremstyle{plain}
\newtheorem{subthm}{Theorem}[subsection]
\theoremstyle{plain}
\newtheorem{subcor}[subthm]{Corollary} 
\theoremstyle{plain}
\newtheorem{sublem}[subthm]{Lemma} 
\theoremstyle{plain}
\newtheorem{subprop}[subthm]{Proposition} 
\theoremstyle{definition}
\newtheorem{subdefn}[subthm]{Definition}
\theoremstyle{remark}
\newtheorem{subrem}[subthm]{Remark}
\theoremstyle{remark}

\theoremstyle{remark}
\newtheorem{subquestion}[subthm]{Question}
\theoremstyle{plain}

\newcommand{\id}{\operatorname{id}}
\newcommand{\Ad}{\operatorname{Ad}}

\newcommand{\Aut}{\operatorname{Aut}}
\newcommand{\Out}{\operatorname{Out}}
\newcommand{\Hom}{\operatorname{\mathbb{H}om}}
\newcommand{\dist}{\operatorname{d}}
\newcommand{\Prob}{\operatorname{Prob}}

\newcommand{\Z}{\mathbb{Z}}
\newcommand{\N}{\mathbb{N}}
\newcommand{\R}{R^\omega}

\newcommand{\p}{\varphi}
\newcommand{\e}{\varepsilon}

\newcommand{\conv}{\operatorname{conv}}
\newcommand{\plane}{\operatorname{plane}}

\newcommand{\IC}{{\mathbb C}}
\newcommand{\IK}{{\mathbb K}}
\newcommand{\IM}{{\mathbb M}}
\newcommand{\IN}{{\mathbb N}}
\newcommand{\cU}{{\mathcal U}}

\newcommand{\IF}{{\mathbb F}}
\newcommand{\IZ}{{\mathbb Z}}
\newcommand{\IB}{{\mathbb B}}
\newcommand{\G}{\Gamma}
\DeclareMathOperator{\Tr}{Tr}

\newcommand{\ip}[1]{\mathopen{\langle}#1\mathclose{\rangle}}

\usepackage{amssymb} \usepackage{amscd}

\makeatother

\begin{document}

\title{Topological dynamical systems associated to II$_1$-factors}

\author{Nathanial P. Brown}

\thanks{Supported by NSF grants DMS-0554870 and DMS-0856197.}


\address{Department of Mathematics, Penn State University, State
College, PA 16802}

\email{nbrown@math.psu.edu}

\begin{abstract}
If $N \subset \R$ is a separable II$_1$-factor, the space $\Hom(N,\R)$ of unitary equivalence classes of unital $*$-homomorphisms $N \to \R$ is shown to have a surprisingly rich structure.  If $N$ is not hyperfinite, $\Hom(N,\R)$ is an infinite-dimensional, complete, metrizeable topological space with convex-like structure, and the outer automorphism group $\Out(N)$ acts on it by ``affine" homeomorphisms.  (If $N \cong R$, then $\Hom(N,\R)$ is just a point.)  Property (T) is reflected in the extreme points -- they're discrete in this case.  For certain free products $N = \Sigma \ast R$, every countable group acts nontrivially on $\Hom( N, \R)$, and we show the extreme points are not discrete for these examples. Finally, we prove that the dynamical systems associated to free group factors are isomorphic.
\end{abstract}

\maketitle

\section{Introduction}

Collections of morphisms between two objects -- e.g., homotopy groups, Pontryagin duals, various representation theories, etc.\ -- are fundamental in mathematics.  In this paper we study topological spaces of morphisms that arise naturally in the context of operator algebras.

\begin{defn}
\label{defn:hom}
Given unital C$^*$-algebras $A$ and $B$, let $\Hom (A, B)$ denote the set of unital $*$-homomorphisms $A \to B$, modulo unitary equivalence.  That is, $[\pi] = [\rho] \in \Hom (A,B)$ if and only if there is a unitary $u \in B$ such that $\pi(a) = u\rho(a)u^*$ for all $a \in A$.
\end{defn}

If one takes $B = B(\mathcal{H})$ and restricts to irreducible representations, we get the spectrum $\hat{A}$ of $A$.  If $B$ is the Calkin algebra (and we use a slightly stronger notion of unitary equivalence), then we get the celebrated Brown-Douglas-Fillmore (BDF) semigroup $\mathrm{Ext}(A)$.  However, for other target algebras $B$ these $\Hom$ spaces have been largely overlooked, despite providing natural and potentially useful invariants.   (For example, it can be shown that if $A$ is a nuclear C$^*$-algebra (cf.\ \cite{me-n-taka}) and $B$ is an ultraproduct II$_1$-factor, then $\Hom (A,B)$ can be identified with the tracial state space $\mathrm{T}(A)$ of $A$; see \cite[Corollary 3.4]{sherman}.)

Note that $\Hom(A,B)$ carries a natural ``topology of point-wise convergence."  That is, $[\pi_n] \to [\pi]$ if there exist representatives $\pi_n'$ of $[\pi_n]$ such that $\pi_n'(a) \to \pi(a)$ for every $a \in A$.  Also note that the outer automorphism group of $A$, $\Out(A) := \Aut(A)/\mathrm{Inn}(A)$, acts on $\Hom(A,B)$ by precomposition, i.e.,  $\alpha.[\pi] := [\pi\circ \alpha^{-1}]$ for all $\alpha \in \Out(A)$ and $[\pi] \in \Hom(A,B)$.  It is easily seen that $[\pi] \mapsto \alpha.[\pi]$ is a  homeomorphism and thus, for fixed $A$, every C$^*$-algebra $B$ gives rise to an invariant -- the topological dynamical system $(\Hom(A,B), \Out(A))$.

Presumably there is a vast general theory to be developed here, but in the present paper we have a specific goal: demonstrate the relevance of these invariants by considering a very special case. Namely, the case that $B$ is an ultraproduct of the hyperfinite II$_1$-factor.

Before proceeding  further, let's fix some notation and conventions. If $(M_n)$ is a sequence of finite factors and $\omega \in \beta \N \setminus \N$ is a free ultrafilter, then we let $(M_n)^\omega$ denote the corresponding ultraproduct  (cf.\ \cite[Appendix A]{me-n-taka}).  It is a II$_1$-factor with unique trace $\tau$ (defined as the limit of traces on the $M_n$'s),  canonical 2-norm $\| x \|_2^2 = \tau(x^*x)$ and unitary group denoted by $\mathcal{U}((M_n)^\omega)$.  If each $M_n$ is isomorphic to the hyperfinite II$_1$-factor $R$, our main case of interest, then we let $\R$ denote the corresponding ultraproduct.

One big advantage of using ultraproducts as targets, and sticking to separable domains, is that in this case \emph{approximate} unitary equivalence is the same thing as unitary equivalence.  That is, if $A$ is (weakly) separable, $\pi, \rho\colon A \to (M_n)^\omega$ are $*$-homomorphisms and there exist unitaries $u_k \in (M_n)^\omega$ such that $\| \pi(a) - u_k \rho(a) u_k^* \|_2 \to 0$, then there is a unitary $u \in (M_n)^\omega$ such that $\pi(a) = u\rho(a)u^*$ for all $a \in A$ (cf.\ \cite[Theorem 3.1]{sherman}).   This fact allows us to define a metric (as opposed to a pseudo-metric, like one gets in BDF theory) on $\Hom(A, (M_n)^\omega)$ as follows.

\begin{defn}
\label{defn:metrics}
If $\{a_n\}$ is a sequence of contractions that generate $A$ (meaning the $*$-algebra they generate is suitably\footnote{Either in norm, or $\sigma$-weakly, depending on whether $A$ is a C$^*$- or W$^*$-algebra.}  dense) and $[\pi], [\sigma] \in \Hom(A, (M_n)^\omega)$, then define $$\dist( [\pi], [\rho] ) = \inf_{u \in \mathcal{U}((M_n)^\omega)} \bigg(\sum_{n = 1}^\infty \frac{1}{2^{2n}} \|\pi(a_n) - u \rho(a_n) u^* \|_2^2 \bigg)^{1/2}.$$
\end{defn}

 The $\ell^2$-formula defining $\dist$ is unimportant, $\ell^p$ variations would work just as well.  For example, if $A$ is generated by finitely many contractions $\{ a_1, \ldots, a_k\}$, one may wish to use an $\ell^\infty$-metric such as $$\dist( [\pi], [\rho] ) = \inf_{u \in \mathcal{U}((M_n)^\omega)} \bigg( \max_{1 \leq i \leq k}  \|\pi(a_n) - u \rho(a_n) u^* \|_2^2 \bigg).$$ In any case, it is easily verified that these types of metrics induce the ``topology of point-wise convergence," as describe above.  Also, an ultraproduct argument shows that the infimum in Definition \ref{defn:metrics} is attained.\footnote{This is probably well known, but here's a sketch in the case $A$ is singly-generated by a contraction $x$ and the metric used is $\dist([\pi],[\rho]) = \inf_{u \in \mathcal{U}((M_n)^\omega)} \| \pi(x) - \Ad u \circ \rho(x) \|_2$. (The general case is similar.) Choose unitaries $u_n \in (M_n)^\omega$ such that $\| \pi(x) - \Ad u_n \circ \rho(x) \|_2 < \dist( [\pi], [\rho] ) + 1/n$.  Lift $\pi(x), \rho(x)$ and the $u_n$'s to elements $(x^\pi_i)_{i \in \N}, (x^\rho_i)_{i \in \N}$ and $(U^{(n)}_i)_{i \in \N} \in \Pi M_n$.  We can arrange that each $U^{(n)}_i$ is a unitary. For each $n\in \N$ we put $$S_n = \{ i \in \N : \| x^\pi_i - U^{(n)}_i x^\rho_i (U^{(n)}_i)^* \|_2 < \dist( [\pi], [\rho] ) + 2/n\}.$$ Now define $(V_i)_{i \in \N}$ by $V_i = 0$ if $i \notin \bigcup_n S_n$; $V_i = U^{(i)}_i$ if $i \in \bigcap_n S_n$; and $V_i = U_i^{(n_i)}$, where $n_i := \max \{ n : i \in S_n\}$, otherwise.  Note that if $i \in \bigcap_n S_n$, then $\| x^\pi_i - V_i x^\rho_i V_i^* \|_2 < \dist( [\pi], [\rho] ) + 2/i$, while $i \in S_m$ for some (but not all) $m$ implies $\| x^\pi_i - V_i x^\rho_i V_i^* \|_2 < \dist( [\pi], [\rho] ) + 2/m$. It follows that $\lim_\omega \| x^\pi _i - V_i x^\rho_i V_i^* \|_2 = \dist( [\pi], [\rho] )$, because for each $\e >0$ the set $S = \{ i \in \N : \| X^\pi _i - V_i X^\rho_i V_i^* \|_2 < \dist( [\pi], [\rho] ) + \e \}$ contains $S_m \cap [m,\infty)$, whenever $2/m < \e$, and $S_m \cap [m,\infty) \in \omega$ for all $m$. }

So, if we fix $A$ and let $(M_n)$ vary among different factors, we get a large class of invariants $(\Hom(A, (M_n)^\omega), \Out(A))$.  There are lots of interesting cases to consider.  For example, the case that each $M_n$ is a finite-dimensional factor, where one could hope to make contact with things like free entropy dimension (cf.\ \cite{V2}), or the case $A = M = M_n$ for all $n$, where classical concepts like central sequences and property $\Gamma$ might be detectable.  But, as mentioned above, we're going to specialize even further to the case that $A$ is a separable II$_1$-factor and $M_n = R$ is the hyperfinite II$_1$-factor.  In fact, to avoid Connes's infamous embedding problem, which asks whether or not $\Hom(N,\R)$ is nonempty for every separable II$_1$-factor $N$, we will further assume that $N$ is  $\R$-embeddable.

\vspace{2mm}

To summarize, in this paper we consider topological dynamical systems associated to separable, $\R$-embeddable, II$_1$-factors; it is a first step, with lots of remaining open questions and other important cases yet to be considered.  But the results so far are encouraging and suggest there is more to be learned.  For example, we will show:
\begin{enumerate}
\item $\Hom(N, \R)$ is always complete (cf.\ the proof of Proposition \ref{prop:convexity}), but almost never compact (see Theorem \ref{thm:main}).

\item $\Hom(N, \R)$ is not a semigroup (like the BDF case), but instead has a convex structure (cf.\ Definition \ref{defn:convex-like} and Proposition \ref{prop:convexity}). This is surprising because we know of no vector-space embedding of $\Hom(N,\R)$.\footnote{However, after a lecture in Nottingham, Ilijas Farah and Aaron Tikuisis suggested different approaches to possible vector-space embeddings.  Aaron's approach essentially builds upon the convex structure established in this paper, but Ilijas's idea may lead to a canonical embedding.  We're looking into it.}

\item In addition to implying contractibility of $\Hom(N, \R)$, hence triviality of most topological invariants, the convex structure implies $\Hom(N, \R)$ has infinite topological dimension whenever $N \subset \R$ is not hyperfinite (see Theorem \ref{thm:main}).\footnote{One motivation for this paper was a theorem of Jung which says that $\Hom(N,\R)$ is a point if and only if $N\cong R$ (\cite{jung}).  An early result of the author showed that if $N \ncong R$, then $\Hom(N,\R)$ is uncountable.  This fact was immediately generalized by Narutaka Ozawa: $\Hom(N,\R)$ is not even second countable when $N \ncong R$ (see Theorem \ref{thm:nonsepA} in the appendix).  Non-second-countability is also a crucial ingredient in proving infinite dimensionality, and the non-compactness result mentioned earlier.}

\item The convex structure allows one to define extreme points.  It turns out that $[\pi] \in \Hom(N, \R)$ is extreme if and only if the relative commutant $\pi(N)' \cap \R$ is a factor (cf. Proposition \ref{prop:extreme}).  It follows that if $N$ has property (T), then the extreme points of $\Hom(N,\R)$ are a discrete subset (i.e., there is a uniform lower bound on the distance between any two of them); see Corollary \ref{cor:T} (and compare with \cite{kenley-dima}).

\item Using the action of $\Out(N)$ and factorial-commutant characterization of extreme points, we give examples where the extreme points are not discrete (see Corollary \ref{cor:nondiscrete}), thereby distinguishing them from the property (T) case.

\item Considering the concrete example $N = L(\mathrm{SL}(3,\Z) \ast R)$, we show in Corollary \ref{cor:proper} that every countable discrete group $\Gamma$ acts on $\Hom(N,\R)$ and there is a particular extreme point with trivial stabilizer (in fact, we get an embedding $\Gamma \hookrightarrow \Hom(N,\R)$ with discrete image).

\item Finally, we prove that the dynamical systems associated to a factor and one of its rescalings are isomorphic (Theorem \ref{thm:scale}).  In particular, the free group factors have isomorphic invariants.
 \end{enumerate}


An outline of this paper is as follows.  In section \ref{sec:convex} we define an abstract notion of convex-like structure on a metric space.  The main result of that section, Theorem \ref{thm:infinitedim}, is that any such space with finite topological dimension must be second countable.  In section \ref{sec:technical} we establish a number of preliminary facts that are needed to define the convex-like structure on $\Hom(N,\R)$.  In section \ref{sec:convex-like} we define the convex-like structure and prove that it satisfies the abstract axioms defined earlier.  Next, in section \ref{sec:extreme}, we study extreme points and consider the property (T) case.  Which is followed, in section \ref{sec:action}, by a discussion of the action of $\Out(N)$ and the free-product type examples mentioned above.  In section \ref{sec:func} we prove that the dynamical systems associated to a factor and one of its corners are always isomorphic, and discuss a number of open problems and related questions.  Finally, in an appendix written by Narutaka Ozawa, it is shown that if $N \subset \R$ and $N \ncong R$, then $\Hom(N,\R)$ is not second countable; and it's shown that every character on a free group induces an automorphism that acts nontrivially on $\Hom(L(\mathbb{F}_n),\R)$.


\begin{acknowledgement*} This paper represents work carried out over several years, and was influenced by numerous questions and comments from various colleagues. It is a pleasure to thank Dietmar Bisch, Jacek Brodzki, Valerio Capraro, Ilijas Farah, Nigel Higson, Kenley Jung, Narutaka Ozawa, Jesse Peterson, Sorin Popa, Dave Sherman, Dima Shlyakhtenko, Andreas Thom, Aaron Tikuisis and anyone else I may have forgotten for helping shape this paper in one way or another.  Also, a significant portion of this work was carried out while on sabbatical at the University of Hawaii and we are eternally  grateful for the hospitality.  \end{acknowledgement*}

\section{Metric spaces with a convex-like structure}
\label{sec:convex}

Let $(X, \dist)$ denote a complete metric space which is bounded, i.e., there is a constant $C$ such that $\dist(x,y) \leq C$ for all $x,y \in X$. Defining an abstract convex-like structure on $X$ is slightly technical, but basically we want a notion of ``convex combination" that enjoys the topological, metric and algebraic properties one would expect if $X$ were an honest convex subset of a bounded ball in some normed linear space.\footnote{Aaron Tikuisis has suggested a slightly different set of axioms, which are probably better for a general theory.  However, the present approach is sufficient for our purposes.} For example,  given $x_1,\ldots,x_n \in X$ and numbers $0\leq t_1,\ldots,t_n \leq 1$ such that $\sum_i t_i = 1$,  we would expect:
\begin{enumerate}
\item(commutativity) $t_1x_1 + \cdots + t_n x_n = t_{\sigma(1)} x_{\sigma(1)} + \cdots + t_{\sigma(n)} x_{\sigma(n)}$ for every permutation $\sigma \in S_n$;

\item(linearity) if $x_1 = x_2$, then $t_1x_1 + t_2 x_2 + \cdots + t_n x_n = (t_1 + t_2) x_1 + t_3 x_3 + \cdots t_n x_n$;

\item(scalar identity) if $t_i = 1$, then $t_1x_1 + \cdots + t_n x_n = x_i$;

\item(metric compatibility) $\|( t_1x_1 + \cdots + t_n x_n ) - (\tilde{t}_1 x_1 + \cdots + \tilde{t}_n x_n) \| \leq C\sum_i | t_i - \tilde{t}_i |$ and $\| ( t_1x_1 + \cdots + t_n x_n ) - ( t_1y_1 + \cdots + t_n y_n )\| \leq \sum t_i \| x_i - y_i \|$;

\item(algebraic compatibility) $$t \bigg( \sum_{i=1}^n t_ix_i \bigg) + (1-t)\bigg( \sum_{j=1}^m \tilde{t}_j \tilde{x}_j \bigg) =  \sum_{i=1}^n t t_ix_i  +  \sum_{j=1}^m (1-t)\tilde{t}_j \tilde{x}_j;$$
\end{enumerate}

To make this idea precise, let $X^{(n)} = X\times  \cdots \times X$ be the $n$-fold Cartesian product and $\Prob_n$ be the set of probability measures on the $n$-point set $\{1,2,\ldots,n\}$, endowed with the $\ell_1$-metric $\| \mu - \tilde{\mu} \| = \sum_{i=1}^n |\mu(i) - \tilde{\mu}(i) |$.

\begin{defn}
\label{defn:convex}
We say $(X,\dist)$ has a \emph{convex-like structure} if for every $n \in \N$ and $\mu \in \Prob_n$ there is a continuous map $\gamma_\mu \colon X^{(n)} \to X$ such that
\begin{enumerate}
\item\label{abelian} for each permutation $\sigma \in S_n$ and $x_1,\ldots, x_n \in X$, $$\gamma_\mu(x_1,\ldots,x_n) = \gamma_{\mu\circ\sigma} (x_{\sigma(1)},\ldots, x_{\sigma(n)});$$

\item\label{linearity} if $x_1 = x_2$, then $\gamma_\mu(x_1,x_2, \ldots,x_n) = \gamma_{\tilde{\mu}} (x_1, x_3, \ldots, x_n)$, where $\tilde{\mu} \in \Prob_{n-1}$ is given by $\tilde{\mu} (1) = \mu(1) + \mu(2)$ and $\tilde{\mu} (j) = \mu(j+1)$ for $2 \leq j \leq n-1$;

\item\label{dirac} if $\mu(i) = 1$, then $\gamma_\mu(x_1,\ldots,x_n) = x_i$;\footnote{In particular, we require the unique element $\mu \in \Prob_1$ to give rise to the identity map on $X$.}

\item\label{metric} there is a constant $C$ such that for all $x_1,\ldots, x_n \in X$, $$\dist (\gamma_{\mu} (x_1,\ldots,x_n), \gamma_{\tilde{\mu}} (x_1,\ldots,x_n)) \leq C \| \mu - \tilde{\mu} \|$$ and for all $y_1,\ldots,y_n \in X$, $$\dist (\gamma_{\mu} (x_1,\ldots,x_n), \gamma_{\mu} (y_1,\ldots,y_n)) \leq \sum_{i = 1}^n \mu(i) \dist(x_i, y_i);$$

\item\label{algebraic} for all $\nu \in \Prob_2$, $\mu \in \Prob_n$, $\tilde{\mu} \in \Prob_m$ and $x_1,\ldots, x_n, \tilde{x}_1,\ldots, \tilde{x}_m \in X$, $$\gamma_{\nu}(\gamma_\mu(x_1,\ldots,x_n), \gamma_{\tilde{\mu}}( \tilde{x}_1,\ldots, \tilde{x}_m) ) = \gamma_{\eta} ( x_1,\ldots,x_n, \tilde{x}_1,\ldots, \tilde{x}_m),$$ where $\eta \in \Prob_{n+m}$ is given by $\eta(i) = \nu(1)\mu(i)$, if $1\leq i \leq n$, and $\eta(j + n) = \nu(2) \tilde{\mu}(j)$, if $1\leq j \leq m$.
\end{enumerate}
\end{defn}

The maps $\gamma_\mu$ are notationally awkward, so for the remainder of this paper we shall revert to the ``convex combination" notation (even though our set $X$ has no additive semigroup structure); that is, if $\sum_i t_i = 1$, $\mu \in \Prob_n$ is defined by $\mu(i) = t_i$ and $x_1,\ldots, x_n \in X$ we will write $$t_1 x_1 + \cdots + t_n x_n  := \gamma_\mu(x_1,\ldots, x_n).$$

The main goal of this section is to show that if $X$ has a convex-like structure and finite topological covering dimension, then $X$ must be second countable. We'll need a few lemmas.

\begin{lem}
\label{lem:wtf}  Let $\{x_k \} \subset X$, $y \in X$ and $t \in (0,1]$ be given.  If $\{ tx_k + (1-t) y\}$ is a convergent sequence, then so is $\{x_k \}$.
\end{lem}

\begin{proof}  If $t=1$, this follows from axiom \ref{dirac} in Definition \ref{defn:convex}.  So, let $$\mathcal{S} = \{ s \in [0,1] : sx_k + (1-s) y \textrm{ is convergent} \}$$ and we'll show $1 \in \mathcal{S}$.

But first we'll show that $\mathcal{S}$ is a closed subset of $[0,1]$.  Given $s_j \in \mathcal{S}$ such that $s_j \to s$ and $\e > 0$, choose $j'$ large enough that $2C| s - s_{j'}| < \e/2$. Then, since $s_{j'} \in \mathcal{S}$, choose $N$ large enough that $\dist ( s_{j'}x_k + (1-s_{j'}) y, s_{j'}x_{l} + (1-s_{j'}) y) < \e / 2$ for all $k,l > N$.  Finally, axiom \ref{metric}  in Definition \ref{defn:convex} yields the following inequalities:
\begin{align*}
\dist(sx_k + (1-s) y, sx_l + (1-s) y) & \leq \dist(sx_k + (1-s) y, s_{j'}x_k + (1-s_{j'}) y) \\
& \hspace{4mm} + \dist ( s_{j'}x_k + (1-s_{j'}) y, s_{j'}x_{l} + (1-s_{j'}) y) \\
& \hspace{4mm} +  \dist(s_{j'}x_l + (1-s_{j'}) y, sx_l + (1-s) y)\\
& \leq 4C| s - s_{j'}| + \dist ( s_{j'}x_k + (1-s_{j'}) y, s_{j'}x_{l} + (1-s_{j'}) y)\\
& < \e,
\end{align*}
for all $k, l > N$. This implies $s \in \mathcal{S}$.

Since $\mathcal{S}$ is closed, $s' := \sup_{s \in \mathcal{S}} s  \in \mathcal{S}$.  Assume $s' < 1$, let $\alpha := \frac{1}{1 + s'}$ and we'll show $\alpha s' + (1-\alpha) \in \mathcal{S}$, thereby contradicting the maximality of $s'$. To do this, first observe that axioms  \ref{algebraic} and \ref{linearity} of Definition \ref{defn:convex} imply that $$(1-\alpha) x_k + \alpha [ s' x_k + (1-s') y]  = [ \alpha s' + (1-\alpha) ] x_k + [\alpha - \alpha s'] y,$$ while axioms  \ref{algebraic}  and \ref{abelian} give the identity $$(1-\alpha) x_k + \alpha [ s' x_l + (1-s') y] = (1-\alpha) x_l + \alpha [ s' x_k + (1-s') y].$$  These algebraic identities yield the following inequalities:
\begin{align*}
\dist([ \alpha s' + (1-\alpha) ] x_k & + [\alpha - \alpha s'] y, [ \alpha s' + (1-\alpha) ] x_l + [\alpha - \alpha s'] y) \\
& = \dist ( (1-\alpha) x_k + \alpha [ s' x_k + (1-s') y], (1-\alpha) x_l + \alpha [ s' x_l + (1-s') y] ) \\
& \leq \dist ( (1-\alpha) x_k + \alpha [ s' x_k + (1-s') y], (1-\alpha) x_k + \alpha [ s' x_l + (1-s') y] )\\
& \hspace{4mm} + \dist ( (1-\alpha) x_l + \alpha [ s' x_k + (1-s') y], (1-\alpha) x_l + \alpha [ s' x_l + (1-s') y]  )\\
& \leq 2\alpha \dist(s' x_k + (1-s') y , s' x_l + (1-s') y] ),
\end{align*}
where  we've used axiom \ref{metric} of Definition \ref{defn:convex} in the last inequality.  This shows $\alpha s' + (1-\alpha) \in \mathcal{S}$, so the proof is complete.
\end{proof}

\begin{defn} Given a finite set $\mathfrak{F} = \{x_1, \ldots, x_n\} \subset X$, let its \emph{convex hull} be the set $$\conv(\mathfrak{F}) = \{ \sum_{i=1}^n \mu(i) x_i : \mu \in \Prob_n \}.$$
\end{defn}

Using the compactness of $\Prob_n$ and axiom \ref{metric} of Definition \ref{defn:convex}, the following fact is easily verified.

\begin{lem} For every finite set $\mathfrak{F} \subset X$, $\conv(\mathfrak{F})$ is sequentially compact (hence compact, since $X$ is a metric space).
\end{lem}

\begin{defn}  Given a finite set $\mathfrak{F} \subset X$, let the \emph{plane generated by $\mathfrak{F}$} be the set $$\plane(\mathfrak{F}) = \{ x \in X : \exists \ y, z \in \conv(\mathfrak{F}), 0 < t \leq 1\textrm{ such that } tx + (1-t) y = z \}.$$
\end{defn}

\begin{lem} \label{lem:plane} For every finite set $\mathfrak{F} \subset X$, $\plane(\mathfrak{F})$ is $\sigma$-compact, hence separable.
\end{lem}

\begin{proof} For each $k \in \N$, let $$\plane_k(\mathfrak{F}) = \{ x \in X : \exists \ y, z \in \conv(\mathfrak{F}), \frac{1}{k} < t \leq 1\textrm{ such that } tx + (1-t) y = z \}.$$  Evidently it suffices to show $\plane_k(\mathfrak{F})$ is sequentially compact, so let $\{ x_n \} \subset \plane_k(\mathfrak{F})$ be an arbitrary sequence.  Since $\conv(\mathfrak{F})$ is compact, we can find a subsequence $n_j \in \N$, real numbers $1/k \leq t, t_{n_j} \leq 1$, and points $y,y_{n_j}, z, z_{n_j} \in \conv(\mathfrak{F})$ such that $$ t_{n_j} x_{n_j} + (1-t_{n_j}) y_{n_j} = z_{n_j},$$ while $t_{n_j} \to t, y_{n_j} \to y$ and $z_{n_j} \to z$.  Metric compatibility (axiom \ref{metric}) implies
\begin{align*}
\dist(tx_{n_j} + (1-t) y, z_{n_j}) & \leq \dist(tx_{n_j} + (1-t) y, tx_{n_j} + (1-t) y_{n_j}) \\
& \hspace{4mm} + \dist ( tx_{n_j} + (1-t) y_{n_j}, t_{n_j}x_{n_j} + (1-t_{n_j}) y_{n_j}) \\
& \leq (1-t)\dist(y, y_{n_j}) + 2C|t - t_{n_j}| \to 0.\\
\end{align*}
Hence, $tx_{n_j} + (1-t) y \to z$, as $j \to \infty$, and so Lemma \ref{lem:wtf} provides us with a point $x \in X$ such that $x_{n_j} \to x$.  The proof will be complete once we observe that $x \in \plane_k(\mathfrak{F})$, but $$\dist(tx + (1-t) y, tx_{n_j} + (1-t) y) \leq t\dist(x, x_{n_j}) \to 0,$$ which implies $tx + (1-t) y = z$, so we're done.
\end{proof}

We're now ready for the main result of this section.

\begin{thm}
\label{thm:infinitedim}
If $(X,d)$ is a bounded, complete metric space with convex-like structure (cf.\ Definition \ref{defn:convex}) and $X$ has finite topological covering dimension, then $X$ is second countable.
\end{thm}

\begin{proof} We'll show the contrapositive by proving that if $X$ isn't second countable, then we can find points $x_0, x_1, x_2, \ldots$ such that for every $n \in \N$, $\conv(\{x_0,\ldots, x_n \})$ contains a homeomorphic copy of the $n$-cube $[0,1]^n$.  The proof is by induction.

$(n = 1)$ Let $x_0 \neq x_1$ be any two distinct points in $X$.  It suffices to show that the map $[0,1] \to X$ given by $t \mapsto tx_0 + (1-t) x_1$ is injective (since this implies it's a homeomorphism onto its image).  Proceeding by contradiction, assume there exist numbers $0 \leq a < b \leq 1$ such that $a x_0 + (1-a) x_1 = b x_0 + (1-b) x_1$.   Consider the set $$\mathcal{S} = \{ t \in [0,1] : t x_0 + (1-t) x_1 = a x_0 + (1-a) x_1 \}.$$ Metric compatibility (axiom \ref{metric}) implies $\mathcal{S}$ is closed, while algebraic compatibility (axiom \ref{algebraic}) implies $\mathcal{S}$ is a convex subset of $[0,1]$. Hence $\mathcal{S} = [a',b']$ for some $a',b' \in [0,1]$.  Since $x_0 \neq x_1$, it can't be the case that $a' = 0$ and $b'=1$, so let's assume $b' < 1$. (We leave the other case to the reader as it is very similar.) Choose a number $0< \alpha < 1$ such that $$a' < \alpha a' + (1-\alpha) 1 < b'.$$  Since $b' < \alpha b + (1 - \alpha)1$, we have that $$(\alpha b' + (1-\alpha)) x_0 + (\alpha - \alpha b') x_1 \neq a x_0 + (1-a) x_1.$$  But this is a contradiction, because algebraic compatibility implies that
\begin{align*}
(\alpha b' + (1-\alpha)) x_0 + (\alpha - \alpha b') x_1 & = \alpha( b' x_0 + (1-b') x_1) + (1-\alpha) x_0 \\
&= \alpha( a' x_0 + (1-a') x_1) + (1-\alpha) x_0\\
& = (\alpha a' + (1-\alpha)) x_0 + (\alpha - \alpha a') x_1\\
& = a x_0 + (1-a) x_1.\\
\end{align*}

(Induction step) Assume we have $x_0, \ldots, x_{n-1} \in X$ such that $\conv(\{x_0, \ldots, x_{n-1} \} )$ contains a homeomorphic copy of $[0,1]^{n-1}$. Since we're assuming $X$ isn't separable, Lemma \ref{lem:plane} ensures that we can find a point $x_n \notin \plane(\{x_0, \ldots, x_{n-1} \} )$.  Define a map $$\conv(\{x_0, \ldots, x_{n-1} \} ) \times [0,1] \to X$$ by $(p,t) \mapsto tp + (1-t) x_n$, and it suffices to show this map is injective on $\conv(\{x_0, \ldots, x_{n-1} \} ) \times (0,1]$ (since its restriction to $[0,1]^{n-1} \times [1/2, 1] \cong [0,1]^n$ will be a homeomorphism).  Again proceeding by contradiction, assume there exist points $(p,a) \neq (q,b)$ such that $ap + (1-a) x_n = bq + (1-b) x_n$.  By the proof of the $n=1$ case, $p \neq q$.

Consider the (nonempty) sets $$\mathcal{S}_p = \{ t \in [0,1] : \exists s \in [0,1] \textrm{ such that } t p + (1-t) x_n = s q + (1-s) x_n \}$$ and $$\mathcal{S}_q = \{ s \in [0,1] : \exists t \in [0,1] \textrm{ such that } t p + (1-t) x_n = s q + (1-s) x_n \}.$$  Metric compatibility implies both $\mathcal{S}_p$ and $\mathcal{S}_q$ are closed, so let $a' := \sup  \mathcal{S}_p \in \mathcal{S}_p$ and $b' := \sup  \mathcal{S}_q \in \mathcal{S}_q$.  Since $p \neq q$ and $x_n \notin \plane(\{x_0, \ldots, x_{n-1} \} )$, $1 \notin \mathcal{S}_p$ and $1 \notin \mathcal{S}_q$, i.e., $a' < 1$ and $b'  < 1$.  Hence, letting $$\alpha = \frac{1- a'}{1 - a' b'} \textrm{  and  } \beta = \frac{1- b'}{1 - a' b'},$$ we have that $\alpha a' + (1-\alpha) \notin  \mathcal{S}_p$ and $\beta b' + (1-\beta) \notin  \mathcal{S}_q$.  This, together with the algebraic identities $$\alpha b' = 1 - \beta, \alpha(1-b') = \beta(1-a') \textrm{ and } 1- \alpha = a' \beta,$$ give our contradiction because algebraic compatibility gives the following equalities:
\begin{align*}
(\alpha a' + (1-\alpha)) p + (\alpha - \alpha a') x_n & = \alpha( a' p + (1-a') x_n) + (1-\alpha) p \\
&= \alpha( b' q + (1-b') x_n) + (1-\alpha) p \\
&= \beta (a' p + (1-a') x_n) + (1- \beta) q\\
&= \beta (b' q + (1-b') x_n) + (1 - \beta) q\\
&= (\beta b' + (1-\beta)) q + (\beta - \beta b') x_n.
\end{align*}
\end{proof}

\section{Technical Facts}
\label{sec:technical}

It turns out that $\Hom(N, \R)$ has a natural convex-like structure, in the sense of Definition \ref{defn:convex}, but proving this requires a number of technical preliminaries.

\subsection{Liftable isomorphisms of corners of $\R$}

It is well-known that all unital endomorphisms of $R$ are approximately inner. (This follows easily from the fact -- essentially due to Murray and von Neumann -- that there is a unique unital embedding of $M_n(\mathbb{C})$ into $R$, up to unitary conjugation.)  It follows that ``liftable" automorphisms of $\R$ are $\aleph_0$-locally inner, i.e., if $\Theta \colon \R \to \R$ lifts to an automorphism of the form $\theta = (\theta_n)_{n \in \N} \in \mathrm{Aut}( \ell^\infty (\N,R))$ (where $\theta_n \in \mathrm{Aut}(R)$), then on every separable subalgebra of $\R$, $\Theta$ is just conjugation by some unitary $u \in \R$ (though $\Theta$ will rarely be inner on all of $\R$ -- see \cite[Theorem 2.5]{sherman}). Here we establish a  technical, but useful extension of this fact for ``liftable" isomorphisms between corners of $\R$.

\begin{sublem}
\label{lem:corner}
Let $p, q \in R$ be projections of the same trace and $\theta\colon pRp \to qRq$ be a unital $*$-homomorphism (i.e.\ $\theta(p) = q$).  Then, there exist partial isometries $v_n \in R$ such that $v_n^* v_n = p, v_n v_n^* = q$ and $\theta(x) = \lim_{n\to \infty} v_n x v_n^*$ for all $x \in pRp$, where the limit is taken in the 2-norm.
\end{sublem}

\begin{proof}  Let $w \in R$ be a partial isometry such that $w^*w = q$ and $ww^* = p$, and consider the unital endomorphisms $\Ad w \circ \theta \colon pRp \to pRp$.  Since $R$ is hyperfinite we can find unitaries $u_n \in pRp$ such that $w\theta(x) w^* = \lim_{n\to \infty} u_n x u^*_n$ for all $x \in pRp$.  Defining $v_n := w^* u_n$ completes the proof.
\end{proof}

The following proposition is an indispensable tool in our analysis. (In particular, we will sometimes cite it when claiming that unital embeddings into $\R$ that differ by ``liftable" isomorphisms of $\R$ are actually unitarily equivalent.)

\begin{subprop}
\label{prop:corner}
Let $p,q \in \R$ be projections of the same trace, $M \subset p\R p$ be a separable von Neumann subalgebra and $\Theta \colon p\R p \to q \R q$ be a unital $*$-homomorphism.  Assume there exist projections $(p_i) , (q_i) \in \ell^\infty(\N, R)$ which are lifts of $p$ and $q$, respectively, such that $\tau_R(p_i) = \tau_R(q_i) = \tau_{\R}(p)$ for all $i \in \N$, and there exist unital $*$-homomorphisms $\theta_i \colon p_i Rp_i \to q_i Rq_i$ such that $(\theta_i(x_i))$ is a lift of $\Theta(x)$, whenever $(x_i) \in \prod p_i Rp_i$ is a lift of $x \in M$. Then, there exists a partial isometry $v \in \R$ such that $v^* v = p, vv^* = q$ and $\Theta(x) = vxv^*$ for all $x \in M$.
\end{subprop}

\begin{proof}  Let us first assume $M = W^*(X)$ is singly generated and let $(x_i) \in \prod p_i Rp_i$ be a lift of $X$. By the previous lemma, we can find partial isometries $v_i \in R$ such that $v_i^*v_i = p_i, v_iv_i^* = q_i$ and $\| \theta_i(x_i) - v_i x_i v_i^* \|_2 < 1/i$.  Clearly $(v_i) \in \ell^\infty(\N, R)$ drops to a partial isometry $v \in \R$ with support $p$ and range $q$; we must check that $\Theta(X) = vXv^*$.  But for every $\e > 0$ the set $$S = \{ i \in \N : \| \theta_i(x_i) - v_ix_i v_i^* \|_2 \}$$ contains the set $\{ n \in \N : n \geq i_0 \}$ for every $i_0 > 1/\e$ -- hence $S \in \omega$, which completes the proof in the singly generated case.

The reader should have no trouble extending to the general case -- simply be more careful when picking the $v_i$'s, arranging inequalities of the form  $\| \theta_i(Y_i) - v_i Y_i v_i^* \|_2 < 1/i$ on a finite set of $Y_i$'s corresponding to lifts of a finite subset of a generating set of $M$.
\end{proof}

\subsection{Commutants of separable subalgebras of $R^\omega$}

Central sequence considerations show that the relative commutant of any separable subalgebra $A \subset \R$ is always large (e.g., diffuse). But we need to push this a bit further.

\begin{subdefn}
\label{defn:amp}
Fix an isomorphism $\theta \colon R\bar{\otimes} R\to R$ and a unital $*$-homomorphism $\pi\colon N\to \R$.  For each $x \in N$, let $(X_i) \in \ell^\infty(\N, R)$ be a lift of $\pi(x)$ and define $1\otimes \pi\colon N \to \R$ by the diagram
$$\begin{CD}
\ell^\infty(\N, R\bar{\otimes}R) @>>> (R\bar{\otimes}R)^\omega\\
@VV\oplus_{\N} \theta V @VV \cong V\\
\ell^\infty(\N, R) @>>> \R
\end{CD}$$
That is, $(1\otimes \pi)(x)$ is the image of the element $(1\otimes X_i) \in \ell^\infty(\N, R\bar{\otimes}R)$ down in $\R$.
\end{subdefn}

\begin{sublem}  $[1\otimes \pi]$ is independent of the isomorphism $\theta$. Moreover, if $[\pi] = [\pi']$ then $ [1\otimes \pi] =[1\otimes \pi']$.
\end{sublem}

\begin{proof} Let $\theta_1, \theta_2\colon R\bar{\otimes} R\to R$ be isomorphisms and $(1\otimes \pi)_1, (1\otimes \pi)_2$ be the resulting embeddings.  Applying Proposition \ref{prop:corner}  to the isomorphisms $\theta_2 \circ \theta_1^{-1}$ the result follows -- indeed, just think about the diagram
$$\begin{CD}
\ell^\infty(\N, R) @>>> \R\\
@VV\oplus_{\N} \theta_1^{-1} V @VV \cong V\\
\ell^\infty(\N, R\bar{\otimes}R) @>>> (R\bar{\otimes}R)^\omega\\
@VV\oplus_{\N} \theta_2 V @VV \cong V\\
\ell^\infty(\N, R) @>>> \R,
\end{CD}$$
let $p = q = 1$ and $\Theta$ be the isomorphism gotten by composing the arrows on the right.

The second statement is obvious since the maps into $(R\bar{\otimes}R)^\omega$ induced by $\pi$ and $\pi'$ are clearly unitarily equivalent.
\end{proof}

\begin{sublem} For every $\pi\colon N \to \R$ we have $[\pi] = [1\otimes \pi]$.
\end{sublem}

\begin{proof}  Fix an isomorphism $\theta \colon R\bar{\otimes}R \to R$ and let $\gamma \colon R \to R\bar{\otimes}R$ be defined by $\gamma(x) = 1\otimes x$. Now apply Proposition \ref{prop:corner}, only this time to the diagram
$$\begin{CD}
\ell^\infty(\N, R) @>>> \R\\
@VV\oplus_{\N} \gamma V @VV \text{unital inclusion}V\\
\ell^\infty(\N, R\bar{\otimes}R) @>>> (R\bar{\otimes}R)^\omega\\
@VV\oplus_{\N} \theta V @VV \cong V\\
\ell^\infty(\N, R) @>>> \R.
\end{CD}$$
\end{proof}

\begin{subrem}[Reformulation]
\label{rem:reformulate}
There is an obvious map $$\R \odot \R \to (R\bar{\otimes}R)^\omega$$ from the \emph{algebraic} tensor product $\R \odot \R$ to $(R\bar{\otimes}R)^\omega$, and it extends to the von Neumann algebraic tensor product $\R \bar{\otimes} \R$ (since the trace on $(R\bar{\otimes}R)^\omega$ evidently restricts to the tensor product trace on $\R \odot \R$).  Thus we have a canonical inclusion $\R \bar{\otimes} \R \subset (R\bar{\otimes}R)^\omega$, and hence any isomorphism $\theta \colon R\bar{\otimes}R \to R$ induces an embedding $\R \bar{\otimes} \R \subset \R$.

This subsection can be summarized as follows: given an embedding $\pi \colon N \to \R$ and an inclusion $\R \bar{\otimes} \R \subset \R$ induced by an isomorphism  $\theta \colon R\bar{\otimes}R \to R$, $$[\pi] = [1\otimes \pi],$$ where $1\otimes \pi \colon N \to 1 \otimes \R \subset \R \bar{\otimes} \R \subset \R$ is the obvious map. (In particular, the class of $1 \otimes \pi$ is independent of $\theta$.)

This point of view -- i.e., considering a nearly canonical embedding $\R \bar{\otimes} \R \subset \R$ -- will be very convenient.
\end{subrem}

\subsection{Cutting by projections in the commutant} If $p \in \R$ is a projection, then the corner $p  \R p$ is isomorphic to $\R$ (i.e., the fundamental group of $\R$ is $\mathbb{R}_+$).  Thus cutting a representation $N \to \R$ by a commuting projection can be viewed as another representation into $\R$.  But how one \emph{views} this -- i.e., how one chooses the isomorphism $p \R p \cong \R$ -- can matter, so we'll stick to nice identifications.

\begin{subdefn}
\label{defn:standard}
Let $p \in R^\omega$ be a projection.  A \emph{standard} isomorphism $\theta_p\colon pR^\omega p \to R^\omega$ is any map gotten in the following way: Lift $p$ to a projection $(p_n) \in \ell^\infty(\N, R)$ such that $\tau_R(p_n) = \tau_{R^\omega}(p)$ for all $n \in \N$, fix isomorphisms $\theta_n\colon p_n Rp_n \to R$ and define $\theta_p$ to be the isomorphism on the right hand side of the commutative diagram
$$\begin{CD}
\ell^\infty(\N, p_nRp_n) @>>> p\R p\\
@VV\oplus \theta_nV @VV\cong V\\
\ell^\infty(\N, R) @>>> \R.
\end{CD}$$
\end{subdefn}

\begin{subdefn}
\label{defn:cutdown}
Define the \emph{cut-down} of $\pi\colon N \to \R$ by a projection $p \in\pi(N)' \cap \R$  to be the map $N \to \R, x \mapsto \theta_p (p\pi(x))$, where $\theta_p$ is a standard isomorphism. The equivalence class of this embedding is independent of the standard isomorpism $\theta_p$ (see the next lemma), hence will be denoted by $[\pi_p]$.
\end{subdefn}

\begin{sublem}
\label{lem:cut-down}
$[\pi_p]$ is independent of $\theta_p$.  If $u \in \R$ is unitary then $[\pi_p] = [\Ad u \circ \pi_{\Ad u(p)}]$.
\end{sublem}

\begin{proof} If $\theta_n, \gamma_n\colon p_n Rp_n \to R$ are sequences of isomorphisms, we can apply Proposition \ref{prop:corner} to the diagram
$$\begin{CD}
\ell^\infty(\N, R) @>>> \R\\
@VV\oplus_{\N} \gamma_n^{-1} V @VV \cong V\\
\ell^\infty(\N, p_nRp_n) @>>> p\R p\\
@VV\oplus \theta_nV @VV\cong V\\
\ell^\infty(\N, R) @>>> \R.
\end{CD}$$
The second assertion also follows from Proposition \ref{prop:corner} and the diagram
$$\begin{CD}
\ell^\infty(\N, R) @>>> \R\\
@VV\oplus_{\N} \theta_n^{-1} V @VV \cong V\\
\ell^\infty(\N, p_nRp_n) @>>> p\R p\\
@VV\oplus \Ad u_nV @VV\Ad u V\\
\ell^\infty(\N, u_np_nRp_nu_n^*) @>>> upu^*\R upu^*\\
@VV\oplus \theta_n\circ\Ad u^*_nV @VV\theta_{upu^*} V\\
\ell^\infty(\N, R) @>>> \R,
\end{CD}$$
where the point is that $\oplus \theta_n\circ\Ad u^*_n$ defines a standard isomorphism $\theta_{upu^*} \colon upu^*\R upu^* \to \R$ which, by the first part, may be used to define $[\Ad u \circ \pi_{\Ad u(p)}]$.
\end{proof}

One might wonder whether $[\pi_p]$ depends on $p$, as opposed to the trace of $p$.  It turns out that it does.

\begin{subprop}
\label{prop:cut-down}
Given $\pi \colon N \to \R$ and  projections $p, q \in \pi(N)' \cap \R$ of the same trace, the following are equivalent:
\begin{enumerate}
\item $[\pi_p] = [\pi_q]$;
\item $p$ and $q$ are Murray-von Neumann equivalent inside $\pi(N)' \cap \R$;
\item There exists $v \in \R$ such that $v^*v = p, vv^* = q$ and $v\pi(x)v^* = q\pi(x)$ for all $x \in N$.
\end{enumerate}
\end{subprop}

\begin{proof}  The equivalence of (2) and (3) is a routine computation left to the reader.

(3) $\Rightarrow$ (1):  As is well-known, we can find lifts $(p_n), (q_n), (v_n) \in \ell^\infty(\N, R)$ of $p,q$ and $v$, respectively, such that $(p_n)$ and $(q_n)$ are projections of trace $\tau(p)$ and each $v_n$ is a partial isometry such that $v_n^*v_n = p_n$ and $v_nv_n^* = q_n$. Now fix isomorphisms $\theta_n \colon p_n R p_n \to R$ and $\gamma_n \colon q_n Rq_n \to R$ and use them to construct the standard isomorphisms defining $\pi_p$ and $\pi_q$, respectively. Finally,
apply Proposition \ref{prop:corner} to the ``liftable" isomorphism on the right hand side of the following diagram
$$\begin{CD}
\ell^\infty(\N, R) @>>> \R\\
@VV\oplus_{\N} \theta_n^{-1} V @VV \cong V\\
\ell^\infty(\N, p_nRp_n) @>>> p\R p\\
@VV\oplus_{\N} \Ad v_n V @VV \Ad v V\\
\ell^\infty(\N, q_nRq_n) @>>> q\R q\\
@VV\oplus \gamma_nV @VV\cong V\\
\ell^\infty(\N, R) @>>> \R.
\end{CD}$$

(1) $\Rightarrow$ (3): Let $(p_n), (q_n) \in \ell^\infty(\N, R)$, $\theta_n \colon p_n R p_n \to R$ and $\gamma_n \colon q_n Rq_n \to R$ be as above and assume there exists a unitary $u \in \R$ such that $\pi_p = \Ad u\circ \pi_q$.  Let $(u_n) \in \ell^\infty(\N, R)$ be a lift of $u$. This time apply Proposition \ref{prop:corner} to the diagram
$$\begin{CD}
\prod_{i\in\N} p_iRp_i @>>> p\R p\\
@VV\oplus_{\N} \theta_n V @VV \cong V\\
\ell^\infty(\N, R) @>>> \R \\
@VV\oplus_{\N} \Ad u_n V @VV \Ad u V\\
\ell^\infty(\N, R) @>>> \R \\
@VV\oplus \gamma_nV @VV\cong V\\
\prod_{i\in\N} q_iRq_i @>>> q\R q,\\
\end{CD}$$
and we get the desired partial isometry $v \in \R$.
\end{proof}

\begin{rem} \label{rem:idontknow} For future reference we note that the implication (3) $\Rightarrow$ (1) above can be generalized (with identical proof) to the case of different embeddings.  That is, if $p \in \pi(N)'\cap \R, q \in \rho(N)'\cap \R$ are projections and there exists a partial isometry $v \in \R$ such that $v^*v = p, vv^* = q$ and $v\pi(x)v^* = q\rho(x)$ for all $x \in N$, then $[\pi_p] = [\rho_q]$.
\end{rem}

Another natural question is whether or not $[\pi] = [\pi_p]$.  The answer is ``sometimes" and we now describe an important case where equality holds.

\begin{subdefn}
\label{defn:tensorproj}
Given an isomorphism $\theta \colon R\bar{\otimes} R\to R$ and a projection $p \in \R$, let $p \otimes 1 \in \R$ be the projection coming from the induced inclusion $\R \bar{\otimes} \R \subset \R$ (as in Remark \ref{rem:reformulate}).\footnote{This notation is misleading, since $p \otimes 1$ depends on $\theta$.  But the same was true in Definition \ref{defn:amp}, and it won't cause any more trouble here than it did there.}
\end{subdefn}

\begin{subprop}
\label{prop:equivcut-down} For every $*$-homomorphism $\pi \colon N \to \R$,  projection $p \in \R$ and  isomorphism $\theta \colon R\bar{\otimes} R\to R$, we have $$[\pi] = [ 1\otimes \pi ] = [(1\otimes \pi)_{p \otimes 1}].$$
\end{subprop}

\begin{proof}  Since both $p \otimes 1$ and $1 \otimes \pi$ arise from the $\theta$-induced inclusion $\R \bar{\otimes} \R \subset \R$ (as in Remark \ref{rem:reformulate}), the maps $1 \otimes \pi$ and $(1\otimes \pi)_{p \otimes 1}$ are more than just unitarily equivalent; if we choose our standard isomorphism carefully (as we're free to do in Definition \ref{defn:cutdown}), $1 \otimes \pi$ and $(1\otimes \pi)_{p \otimes 1}$ are actually equal -- on the nose.

More precisely, we can choose a standard isomorphism $\theta_{p\otimes 1} \colon (p\otimes 1) (R \bar{\otimes} R)^\omega (p\otimes 1) \to (R \bar{\otimes} R)^\omega$ in such a way that when restricted to the canonical subfactor $\R \bar{\otimes} \R \subset (R \bar{\otimes} R)^\omega$ it looks like $\theta_p \otimes \id \colon (p\R p) \bar{\otimes} \R \to \R \bar{\otimes} \R$, for some standard isomorphism $\theta_p \colon p \R p \to \R$. (We leave the details to the reader, but $\theta_{p\otimes 1}$ is constructed from isomorphisms of the form $(qRq) \bar{\otimes} R \cong R \bar{\otimes} R$ that leave the right tensor factor alone.)  Evidently one then has $$(1\otimes \pi)_{p \otimes 1}(x) := \theta_{p\otimes 1} \big( (p\otimes 1)(1\otimes \pi)(x) \big) = (1\otimes \pi)(x),$$ for all $x \in N$. In view of Remark \ref{rem:reformulate}, this completes the proof.
\end{proof}

\section{A convex-like structure on $\Hom(N, \R)$}
\label{sec:convex-like}

It's easy enough to imagine a convex-like stucture on $\Hom(N, \R)$.  Namely, given $*$-homomorphisms $\pi, \rho \colon N \to \R$ and $0 < t < 1$, take a projection $p_t \in (\pi(N) \cup \rho(N))' \cap \R$ such that $\tau(p_t) = t$ and define the convex combination $t\pi + (1-t)\rho$ to be $$x \mapsto \pi(x)p_t + \rho(x) p_t^{\perp}.$$ Unfortunately this procedure isn't well defined on classes in $\Hom(N, \R)$, so we have to be a bit more careful.

\begin{defn}
\label{defn:convex-like}
Given $[\pi_1], \ldots, [\pi_n] \in \Hom(N,\R)$ and $t_1,\ldots, t_n \in [0,1]$ such that $\sum t_i = 1$, we define $$\sum_{i=1}^n t_i [\pi_i] := [ \sum_{i=1}^n \theta_i^{-1}\circ\pi_i],$$ where $\theta_i \colon p_i \R p_i \to \R$ are standard isomorphisms and $p_1, \ldots, p_n \in \R$ are orthogonal projections such that $\tau(p_i) = t_i$ for $i \in \{1,\ldots, n\}$.
\end{defn}

\begin{lem} $\sum_{i=1}^n t_i [\pi_i]$ is well defined, i.e., independent of the projections $p_i$, the standard isomorphisms $\theta_i$ and the representatives $\pi_i$.
\end{lem}

\begin{proof}  Assume $\sigma_i \colon q_i \R q_i \to \R$ are standard isomorphisms, where the $q_i$'s are orthogonal projections of trace $t_i$, and $[\rho_i] = [\pi_i]$.  It suffices to show the existence of partial isometries  $v_i \in \R$ such that $v_i^*v_i = p_i, v_i v_i^* = q_i$ and $$v_i \theta_i^{-1} \circ \pi_i (x) v_i^* = \sigma_i^{-1} \circ \pi_i (x)$$ for all $x \in N$. Indeed, because $\rho_i = \Ad u_i \circ \pi_i$ for some unitaries $u_i$, one easily checks that  $u := \sum \sigma_i^{-1}(u_i)v_i$ is a unitary conjugating $\sum_{i=1}^n \theta_i^{-1}\circ\pi_i$ over to $\sum_{i=1}^n \sigma_i^{-1}\circ\rho_i$.

The existence of the desired partial isometries follows easily from Proposition \ref{prop:corner} (applied to the isomorphism $\sigma_i^{-1} \circ \theta_i \colon p_i \R p_i \to q_i \R q_i$) and the definition of standard isomorphism.
\end{proof}

Before proving that this definition of convex combination satisfies the axioms of Definition \ref{defn:convex}, perhaps a couple examples are in order.  The first illustrates the intuitive nature of our definition, while the second illustrates its flexibility (hence utility).

\begin{example}
\label{example:extreme}
Assume $p \in \pi(N)' \cap \R$ is a projection.  Then it is immediate from the definitions that $$[\pi] = \tau(p) [\pi_p] + \tau(p^{\perp}) [\pi_{p^\perp}].$$
\end{example}

Our next example requires some notation.

\begin{defn} Given $\pi \colon N \to \R$, an isomorphism  $\theta \colon R\bar{\otimes} R\to R$ and a projection $p \in \R$, let $p \otimes \pi \colon N \to (p \otimes 1) \R (p \otimes 1)$ be the representation $x \mapsto 1\otimes \pi(x) (p \otimes 1)$, where $p \otimes 1$ was defined in Definition \ref{defn:tensorproj}.\footnote{Note that $p \otimes \pi$ is not the same as $(1\otimes \pi)_{p\otimes 1}$ since these maps have different ranges.  On the other hand, they only differ by a standard isomorphism.}
\end{defn}

\begin{example}
\label{example:tensor}
Given an isomorphism  $\theta \colon R\bar{\otimes} R\to R$, $[\pi_1], \ldots, [\pi_n] \in \Hom(N,\R)$ and $t_1,\ldots, t_n \in [0,1]$ such that $\sum t_i = 1$, we have that $$\sum_{i=1}^n t_i [\pi_i] := [ \sum_{i=1}^n p_{i}\otimes \pi_i],$$ where $p_i \in \R$ are pairwise orthogonal projections of trace $t_i$ and the projections $\{ p_{t_i} \otimes 1\}$ and representations $\{ 1 \otimes \pi_i \}$ are all defined in terms of the $\theta$-induced inclusion  $\R \bar{\otimes} \R \subset \R$.

To see this one needs to verify that $p_i \otimes \pi_i$ can be identified with $\theta_i^{-1} \circ (1 \otimes \pi_i)$ for an appropriate standard isomorphism $(p_i \otimes 1) (R \bar{\otimes} R)^\omega (p_i \otimes 1) \to (R \bar{\otimes} R)^\omega$.  This, however, is easy if one considers isomorphisms of the form $\sigma \otimes \id_R \colon pRp \bar{\otimes} R \to R \bar{\otimes} R$.
\end{example}

\begin{prop}
\label{prop:convexity} For any separable II$_1$-factor $N$, $\Hom(N, \R)$ has a convex-like structure in the sense of Definition \ref{defn:convex}.
\end{prop}

\begin{proof}  It is clear that $\Hom(N,\R)$ is bounded and we already observed that $\dist$ is a metric, so let's verify completeness.

Let $[\pi_n]$ be a Cauchy sequence. If $X = \{X_1, X_2, \ldots\}$ is a generating set for $N$, then a routine exercise produces unitaries $u_n \in \R$ such that $\{ \Ad u_n\circ \pi_n (X_j) \}$ is 2-norm Cauchy for every $j \in \N$. By completeness of the unit ball of $\R$ in the 2-norm, we can find operators $Y_j \in \R$ such that $\| \Ad u_n\circ \pi_n (X_j) - Y_j\|_2 \to 0$ for all $j \in \N$. It follows that the $*$-moments of $\{Y_1, Y_2, \ldots\}$ agree with those of $\{X_1, X_2, \ldots\}$, hence there is an embedding $\pi\colon N \to W^*(Y_1,Y_2,\ldots)$ such that $\pi(X_j) = Y_j$ and it's easily checked that $\dist([\pi_n], [\pi]) \to 0$.

Regarding the five axioms of Definition \ref{defn:convex}, the first three are easy and will be left to the reader.  (The picture provided by Example \ref{example:tensor} makes the second axiom transparent.)

To verify the first part of axiom \ref{metric}, fix $[\pi_1], \ldots, [\pi_n] \in \Hom(N,\R)$ and two sets of nonnegative numbers $\{t_1,\ldots,t_n\}$ and $\{t_1',\ldots,t_n'\}$ that add up to one.  Choose orthogonal projections $\tilde{p}_i \in \R$ such that $$\tau(\tilde{p}_i) = \min \{t_i, t_i' \}$$ and let $Q = 1 - \sum_{i=1}^n \tilde{p}_i$.  Next, choose projections $q_i \leq Q$ such that $$\tau(q_i) = t_i - \min \{t_i, t_i' \}.$$  Since $1 = \sum t_i = \sum (t_i - \min \{t_i, t_i' \}) + \sum \min \{t_i, t_i' \}$, we have that $\tau(Q) = \sum (t_i - \min \{t_i, t_i' \})$.  Thus we can choose the $q_i$'s to be pairwise orthogonal.   Similarly, we can find pairwise orthogonal projections $q_i' \leq Q$ such that $$\tau(q_i') = t_i' - \min \{t_i, t_i' \}.$$  Now define projections $$p_i := \tilde{p}_i + q_i \textrm{ and } p_i' := \tilde{p}_i + q_i'.$$  Observe that $\tau(p_i) = t_i$ and $\tau(p_i') = t_i'$ for all $i \in \{1,\ldots,n\}$, while the cancellation built into the definitions ensures that $\| p_i - p_i' \|_2 = |t_i - t_i'|$.  Hence, in view of Example \ref{example:tensor}, we get the following estimate:
\begin{align*}
\dist( \sum t_i [\pi_i], \sum t_i' [\pi_i] ) & = \dist( \big[\sum p_i \otimes \pi_i\big],  \big[\sum p_i' \otimes \pi_i\big]) \\
&\leq \bigg( \sum_j \frac{1}{2^{2j}} \| \sum_i p_i \otimes \pi_i (X_j) - p_i' \otimes \pi_i (X_j) \|_2^2 \bigg)^{1/2}\\
&\leq \bigg( \sum_j \frac{1}{2^{2j}} ( \sum_i \| p_i  - p_i'  \|_2 )^2 \bigg)^{1/2}\\
& \leq \sum_i \| p_i  - p_i'  \|_2\\
& = \sum |t_i - t_i'|.
\end{align*}
For the other inequality in axiom \ref{metric}, we keep the numbers $\{t_1,\ldots,t_n\}$ and projections $\{ p_1, \ldots, p_n\}$ as above, but let $[\rho_1], \ldots, [\rho_n] \in \Hom(N, \R)$ and $\e > 0$ be arbitrary. Choose unitaries $u_i \in \R$ such that $$ \bigg( \sum_j \frac{1}{2^{2j}} \| \pi_i (X_j) - u_i \rho_i (X_j)u_i^* \|_2^2 \bigg)^{1/2}  \leq \dist([\pi_i], [\rho_i]) + \e.$$  Define a unitary by $U := \sum_i p_i \otimes u_i$ and we have
\begin{align*}
\dist( \sum t_i [\pi_i], \sum t_i [\rho_i] ) & = \dist( \big[\sum p_i \otimes \pi_i\big],  \big[\sum p_i \otimes \rho_i\big]) \\
&\leq \bigg( \sum_j \frac{1}{2^{2j}} \| \sum_i p_i \otimes \pi_i (X_j) - U(\sum p_i \otimes \rho_i (X_j))U^* \|_2^2 \bigg)^{1/2}\\
&\leq \bigg( \sum_j \frac{1}{2^{2j}} \| \sum_i p_i \otimes (\pi_i (X_j) - u_i \rho_i (X_j) u_i^*) \|_2^2 \bigg)^{1/2}\\
& =  \bigg( \sum_j \frac{1}{2^{2j}}  \sum_i \|p_i\|_2^2 \|\pi_i (X_j) - u_i \rho_i (X_j) u_i^* \|_2^2 \bigg)^{1/2}\\
& \leq \bigg( \sum_i  \|p_i\|_2^2 (\dist([\pi_i], [\rho_i]) + \e)^2 \bigg)^{1/2}\\
&\leq \sum_i t_i (\dist([\pi_i], [\rho_i]) + \e).
\end{align*}

To verify axiom \ref{algebraic} we fix $0 \leq s, t_i, t_j' \leq 1$, where $1 \leq i \leq n$, $1 \leq j \leq m$ and $\sum_i t_i = \sum_j t_j' = 1$, and fix $[\pi_1], \ldots, [\pi_n], [\rho_1],\ldots,[\rho_m] \in \Hom(N,\R)$; we must show $$s \bigg( \sum_i t_i [\pi_i] \bigg) + (1-s) \bigg( \sum_j t_j' [\rho_j] \bigg) = \sum_i st_i [\pi_i] +  \sum_j (1-s)t_j' [\rho_j].$$ So, choose sets of orthogonal projections $\{ p_1, \ldots, p_n\}, \{p_1',\ldots, p_m' \} \subset \R$ such that $\tau(p_i) = t_i$ and $\tau(p_j') = t_j'$, fix standard isomorphisms $\theta_i \colon p_i \R p_i \to \R, \theta_j'\colon p_j' \R p_j' \to \R$, and pick another projection $q \in \R$ of trace $s$ and two standard isomorphisms $\sigma \colon q \R q \to \R, \sigma^{\perp} \colon q^{\perp} \R q^{\perp} \to \R$. The key observation is that $$\theta_i \circ \sigma|_{q \sigma^{-1}(p_i) \R q \sigma^{-1}(p_i)} \colon q \sigma^{-1}(p_i) \R q \sigma^{-1}(p_i) \to \R$$ is a standard isomorphism (and similarly for $\sigma^{\perp}$ and the $\theta_j'$'s), hence we can use $\sigma^{-1}\circ\theta_i^{-1}$ in the definition of $\sum_i st_i [\pi_i] +  \sum_j (1-s)t_j' [\rho_j]$.  Now one computes
\begin{align*}
s \bigg( \sum_i t_i [\pi_i] \bigg) + (1-s) \bigg( \sum_j t_j' [\rho_j] \bigg) &= [ \sigma^{-1}\bigg( \sum_{i=1}^n \theta_i^{-1}\circ\pi_i \bigg) + (\sigma^{\perp})^{-1}\bigg( \sum_{j=1}^m (\theta_i')^{-1}\circ\pi_i' \bigg) ]\\
& = [\sum_{i=1}^n \sigma^{-1}\circ\theta_i^{-1}\circ\pi_i  + \sum_{j=1}^m (\sigma^{\perp})^{-1}\circ(\theta_i')^{-1}\circ\pi_i' ]\\
& = \sum_i st_i [\pi_i] +  \sum_j (1-s)t_j' [\rho_j].
\end{align*}
\end{proof}


Here is a consequence of our work so far. Below, $| \cdot |$ denotes cardinality, $c$ is the cardinality of the continuum and $\mathrm{dim}(\cdot)$ denotes the topological (i.e., Lebesgue) covering dimension.

\begin{thm}
\label{thm:main}
Let $N \subset \R$ be a separable II$_1$-factor. The following are equivalent:
\begin{enumerate}
\item\label{1} $N \cong R$;

\item\label{2} $|\Hom(N,\R)|  = 1$;

\item\label{3} $|\Hom(N,\R)| < c$;

\item\label{4} $\Hom(N, \R)$ is second countable;

\item\label{5} $\Hom(N, \R)$ is compact;

\item\label{6} $\Hom(N, \R)$ contains an isolated point;

\item\label{7} $\mathrm{dim}(\Hom(N, \R)) < \infty$;

\end{enumerate}
In particular, if $N \subset \R$ but $N \ncong R$, then $\Hom(N,\R)$ is an infinite-dimensional, non-separable, complete metric space with convex-like structure.
\end{thm}

\begin{proof}  The equivalence of (\ref{1}) and (\ref{2}) is due to Jung (cf., \cite{jung}), while Ozawa's nonseparability result (see Theorem \ref{thm:nonsepA} in the appendix) implies (\ref{1}) is also equivalent to (\ref{3}) and (\ref{4}).  Clearly $(\ref{2}) \Longrightarrow (\ref{5})$, and the fact that $\dist$ is a metric gives the implication $(\ref{5})\Longrightarrow (\ref{4})$.  Proposition \ref{prop:convexity} implies (\ref{2}) and (\ref{6}) are equivalent since $\Hom(N,\R)$ is contractible (to any one of its points). Finally, Theorem \ref{thm:infinitedim} and Proposition \ref{prop:convexity} give the implication $(\ref{7})\Longrightarrow (\ref{4})$ and, since $(\ref{2})\Longrightarrow (\ref{7})$ is trivial, this completes the proof.
\end{proof}

\section{Extreme points of $\Hom(N,\R)$}
\label{sec:extreme}

Having a convex-like structure, it is natural to look for extreme points in $\Hom(N,\R)$.  It turns out that they have an elegant characterization.  But first, a useful observation.

\begin{rem}[Cutting convex combinations]  One can recover the points in a convex combination by  cutting with the right projections.  That is, if $[\pi] = \sum t_i [\pi_i]$, then for each $i$ there is a projection $q_i \in \pi(N)' \cap \R$ such that $[\pi_{q_i}] = [\pi_i]$.  This is immediate if $\pi$ happens to be the representative $$\sum \theta_i^{-1} \circ \pi_i,$$ and follows from Lemma \ref{lem:cut-down} for all other representatives.
\end{rem}

\begin{prop}
\label{prop:extreme} Given $[\pi] \in \Hom(N,\R)$, the following are equivalent:
\begin{enumerate}
\item\label{prop:1} $[\pi]$ is an extreme point (i.e., can't be written as a nontrivial convex combination);

\item\label{prop:2} $\pi(N)' \cap \R$ is a factor;

\item\label{prop:3} $[\pi] = [\pi_p]$ for every nonzero projection $p \in \pi(N)' \cap \R$.
\end{enumerate}
\end{prop}

\begin{proof}
$(\ref{prop:1}) \Longrightarrow (\ref{prop:3})$: In light of Example \ref{example:extreme}, this is trivial.

$(\ref{prop:3}) \Longrightarrow (\ref{prop:2})$: Since $\pi(N)' \cap \R$ is diffuse (in fact, contains a copy of $\R$), it suffices to show that two projections in $\pi(N)' \cap \R$ are Murray-von Neumann equivalent  in $\pi(N)' \cap \R$ if and only if they have the same trace. But since $[\pi_p] = [\pi_q] (= [\pi])$ for all projections $p,q \in \pi(N)' \cap \R$, this follows from Proposition \ref{prop:cut-down}.

$(\ref{prop:2}) \Longrightarrow (\ref{prop:1})$:  Assume $[\pi] = t [\rho] + (1-t) [\sigma]$ and choose a projection $p \in \pi(N)' \cap \R$ such that $[\pi_p] = [\rho]$.  Next, choose a projection $q \in \pi(N)' \cap \R$ such that $\tau(p) = \tau(q)$ and $[\pi_q] = [\pi]$ (cf.\ Lemma \ref{lem:cut-down} and Proposition \ref{prop:equivcut-down}).  Since $\pi(N)' \cap \R$ is a factor, $p$ and $q$ are Murray-von Neumann equivalent and hence Proposition \ref{prop:cut-down} implies that $[\rho] = [\pi]$.  A  similar argument shows $[\pi] = [\sigma]$, so $[\pi]$ is an extreme point.
\end{proof}

Though we won't need it, here's a cute consequence.

\begin{cor} $R$ is the unique separable II$_1$-factor with the property that every embedding into $\R$ has factorial commutant.
\end{cor}

\begin{proof} If $N \ncong R$, then by Jung's result \cite{jung} there are at least two distinct elements in $\Hom(N, \R)$.  Every point on the line segment joining distinct points will not be extreme, hence won't have a factorial commutant.
\end{proof}

Another cute consequence occurs at the other end of the amenability spectrum.

\begin{cor}
\label{cor:T}
If $N$ has property (T) (see \cite[Definition 12.1.6]{me-n-taka}), then the extreme points of $\Hom(N, \R)$ form a discrete subset.
\end{cor}

\begin{proof}  A systematic treatment of rigid (sub)factors can be found in Popa's seminal notes \cite{popa:correspondences}.  Though formulated in our language, Popa established the following fact in \cite[Section 4.5]{popa:correspondences} (or see \cite[Theorem 1.1]{kenley-dima} for a published version): for every $\e > 0$, there is a $\delta > 0$ such that if $[\pi],[\rho] \in \Hom(N,\R)$ and $\dist([\pi], [\rho]) < \delta$, then there are projections $p \in \pi(N)' \cap \R$, $q \in \rho(N)' \cap \R$ and a partial isometry $v \in \R$ such that $v^* v = p, vv^* = q, \tau(p) > 1 - \e$ and $v \pi(x) v^* = q\rho(x)$ for all $x \in N$.  In particular, $[\pi_p] = [\rho_q]$ (see Remark \ref{rem:idontknow}).

So, to see why the extreme points are a discrete subset, let $\e > 0$ be given, choose $\delta$ as above and assume that $[\pi]$ and $[\rho]$ are extreme points with $\dist([\pi], [\rho]) < \delta$.  Then there are nonzero projections $p \in \pi(N)' \cap \R$, $q \in \rho(N)' \cap \R$ such that $[\pi_p] = [\rho_q]$.   By Proposition \ref{prop:extreme}, this implies $[\pi] = [\rho]$.
\end{proof}

We will see in the next section that the extreme points of $\Hom(N,\R)$ are not discrete for many natural examples arising from free products.

\section{The action of $\Out(N)$ on $\Hom(N,\R)$}
\label{sec:action}

Having fleshed out the structure of $\Hom(N,\R)$, we now observe that the outer automorphism group $\Out(N)$ acts by ``affine" homeomorphisms on this space.  We then present lots of examples where the action is nontrivial.\footnote{We thank Sorin Popa for pointing out that similar ideas were used in \cite[Chapter 4]{popa:correspondences}.}

\begin{defn}
\label{defn:action} Define an action of $\Out(N)$ on $\Hom(N,\R)$ as follows: given $\alpha \in \Out(N)$ and $[\pi ] \in \Hom(N,\R)$, define $\alpha.[\pi] := [\pi\circ \alpha^{-1}].$
\end{defn}

It is clear that $\alpha.[\pi]$ is well defined, and easy to check that $[\pi] \mapsto [\pi\circ \alpha^{-1}]$ is a continuous map.  Moreover, with notation as in Definition \ref{defn:convex-like}, $$\alpha.\bigg( \sum t_i [\pi_i] \bigg) = \alpha.\bigg[ \sum \theta_i^{-1}\circ \pi_i \bigg] = \bigg[ \sum \theta_i^{-1}\circ (\pi_i \circ \alpha^{-1}) \bigg] = \sum t_i \alpha.[\pi_i],$$ and hence the convex-like structure is preserved as well.

For an embarrassingly long time, the author could not find an example where $\Out(N)$ acts nontrivially. The question amounts to this: given $\alpha \in \Out(N)$ can one find an embedding $\pi\colon N \to \R$ such that $\alpha$ \emph{does not} extend to an inner automorphism of $\R$ (since $[\pi\circ\alpha] = [\pi] \Longleftrightarrow$ there is a unitary such that $\pi(\alpha(x)) = u\pi(x)u^*$)?

During a fruitful visit to UCLA in June 2010, we put this question to Dima Shlyaktenko. He explained how free entropy calculations suggest that non-extendability ought to be generic for certain automorphisms of free group factors, and proceeded to outline a proof.\footnote{In more detail, if you consider automorphisms of $L(\mathbb{F}_2) \ast L(\mathbb{F}_2)$ of the form $\id \ast \alpha$, then ``most" embeddings of $L(\mathbb{F}_2) \ast L(\mathbb{F}_2)$, at least into matrix ultraproducts, should have the property that $\id \ast \alpha$ can't be extended to an inner automorphism. Roughly, this should be the case because if $u$ implements $\id \ast \alpha$, then it commutes with the left copy of $L(\mathbb{F}_2)$ and microstate spaces of commuting unitaries aren't large enough to account for all possible embeddings of $L(\mathbb{F}_2) \ast L(\mathbb{F}_2)$.}  This free-probabalistic heuristic changed our perspective dramatically, leading to the following simple lemma and the examples which follow.  (Thanks Dima!)

\begin{lem}
\label{lem:neq}
Let $\pi \colon N \to \R$ be an embedding and assume there exists a subalgebra $\Sigma \subset N$ such that $\pi(\Sigma)'\cap \R = \pi(N)' \cap \R$.  Assume $\alpha \in \Aut(N)$ is nontrivial, but restricts to the identity on $\Sigma$.  Then $\alpha.[\pi] \neq [\pi]$.
\end{lem}

\begin{proof} If $u \in \R$ satisfies the equation $\pi(\alpha(x)) = u\pi(x) u^*$ for all $x \in N$, then $u \in \pi(\Sigma)'\cap \R$.  But this forces $u$ to commute with all of $\pi(N)$, contradicting the fact that $\alpha$ is nontrivial.
\end{proof}

Inspired by this lemma, we now construct a large class of examples where $\Hom$ has copious extreme points and automorphisms act nontrivially.  (See Theorem \ref{thm:charaction} for more examples, due to Ozawa.) We begin with finite-dimensional irreducible representations of a property (T) group, just as Voiculescu and Wassermann did in \cite{V3} and  \cite{W}, respectively.

\begin{defn}
\label{defn:sigma} If $\Gamma$ is a discrete group with Kazhdan's property (T) and $\sigma_n \colon \Gamma \to \IM_{k(n)}(\IC)$ are \emph{irreducible} unitary representations, then the direct sum $\oplus \sigma_n \colon \Gamma \to \prod \IM_{k(n)}(\IC)$ descends to a unitary representation of $\Gamma$ into the matrix ultraproduct and we let $$\Sigma(\Gamma) \subset (\IM_{k(n)}(\IC))^\omega$$ be the von Neumann algebra\footnote{Actually $\Sigma(\Gamma)$ is a factor, but we won't need this.} generated by this representation.
\end{defn}

Now we fix $\sigma_n \colon \Gamma \to \IM_{k(n)}(\IC)$ and consider the von Neumann algebra $W^*(\Sigma(\Gamma), \{Y_i \})$ generated by $\Sigma(\Gamma)$ and an arbitrary sequence of contractions $\{Y_i \} \subset (\IM_{k(n)}(\IC))^\omega$.  This is the algebra to which we will apply Lemma \ref{lem:neq}, but a few more preliminaries are needed.

\begin{defn}
\label{defn:iota} There is an obvious inclusion $$(\IM_{k(n)}(\IC))^\omega \cong (\IM_{k(n)}(\IC)\otimes 1)^\omega \subset (\IM_{k(n)}(\IC)\otimes R)^\omega$$ and we let $\iota \colon W^*(\Sigma(\Gamma), \{Y_i \}) \to(\IM_{k(n)}(\IC)\otimes R)^\omega$ denote the restriction of this map to $W^*(\Sigma(\Gamma), \{Y_i \})$.
\end{defn}

We started with irreducible representations so we could control commutants. The next lemma, which shows $\iota \colon W^*(\Sigma(\Gamma), \{Y_i \}) \to(\IM_{k(n)}(\IC)\otimes R)^\omega$ satisfies the hypotheses of Lemma \ref{lem:neq}, is a routine exercise and will be left to the reader.\footnote{Use the fact that a projection commuting with  $\iota(\Sigma(\Gamma))$ can be lifted to a sequence of projections that almost commute with $\sigma_n(\Gamma) \otimes 1$; hence can be perturbed to honestly commuting projections.  Irreducibility of the $\sigma_n$'s forces the lifts into $1 \otimes R$.}

\begin{lem}
\label{lem:commute} We have $$\iota(\Sigma(\Gamma))'\cap (\IM_{k(n)}(\IC)\otimes R)^\omega = \iota(W^*(\Sigma(\Gamma), \{Y_i \}))' \cap (\IM_{k(n)}(\IC)\otimes R)^\omega = (1\otimes R)^\omega.$$ In particular, by Proposition \ref{prop:extreme}, $[\iota]$ is an extreme point.
\end{lem}

Recall that property (T) groups come with critical sets and Kazhdan constants, so let's fix a critical set $F \subset \Gamma$ and Kazhdan constant $\kappa > 0$.\footnote{This means $F$ is a finite set with the property that for every unitary representation $U \colon \Gamma \to B(\mathcal{H})$ and $\e > 0$, if $v \in \mathcal{H}$ and $\| U_g(v) - v \| \leq \epsilon$ for all $g \in F$, then there exists $v_0 \in \mathcal{H}$ such that $U_s(v_0) = v_0$ for all $s \in \Gamma$ and $\| v - v_0 \| \leq \e / \kappa$.}

\begin{lem}
\label{lem:inequality} For every contraction $X \in \iota(W^*(\Sigma(\Gamma), \{Y_i \}))$ and unitary $u \in (\IM_{k(n)}(\IC)\otimes R)^\omega$, $$\| X - uXu^* \|_2 \leq \frac{2}{\kappa} \max_{g \in F} \| \iota(\Sigma_g) - u\iota(\Sigma_g) u^* \|_2,$$ where $\Sigma_g \in \Sigma(\Gamma)$ denotes the image of a group element $g \in \Gamma$.
\end{lem}

\begin{proof} If $V\colon \Gamma \to M$ is a unitary representation into a finite von Neumann algebra with faithful trace $\tau$, then $m \mapsto V_g m V_g^*$ extends to a unitary representation of $\Gamma$ on $L^2(M,\tau)$.  Hence, if $\max_{g \in F} \| V_g m V_g^* - m \| \leq \e,$ then we can find $m_0 \in \{ V_s \}_{s \in \Gamma}' \cap M$ such that $\| m - m_0 \|_2 \leq \e / \kappa$.  In the case that $m$ is a unitary, we have $\| 1 - m_0m^* \|_2 \leq \e / \kappa$, too.\footnote{Incidentally, this well-known general fact implies $\Sigma(\Gamma)$ is a factor, as well as Lemma \ref{lem:commute}.}

Applying this general fact to $u \in (\IM_{k(n)}(\IC)\otimes R)^\omega$ and $\tilde{\e} := \max_{g \in F} \| \iota(\Sigma_g) - u\iota(\Sigma_g) u^* \|_2$, we can find $u_0 \in (1\otimes R)^\omega$ (cf.\ Lemma \ref{lem:commute}) such that $\| u - u_0 \|_2 = \| 1 - u_0u^* \|_2 \leq \tilde{\e} / \kappa$. Since $u_0$ commutes with $X$, we have $$\| X - uXu^* \|_2 = \| X(1 - u_0u^*) + (u_0 - u)Xu^*\|_2 \leq \|1 - u_0u^* \|_2 + \| u_0 - u \|_2 \leq \frac{2}{\kappa} \tilde{\e},$$ as claimed.
\end{proof}


So far we haven't worried much about the metric $\dist$, but there is a particularly convenient one in the present context.  Indeed, $W^*(\Sigma(\Gamma), \{Y_i \})$ has a nice generating set, namely a critical set $F$ for $\Gamma$ together with $\{Y_i \}$.  Though the following formula is a little different from what we're used to, it is easily seen to be equivalent to a standard metric coming from $F \cup \{Y_i \}$.

\begin{defn}
\label{defn:delta} Define a metric $\dist^*$ on $\Hom( W^*(\Sigma(\Gamma), \{Y_i \}), \R)$ by $$\dist^* ([\pi], [\rho] ) := \inf_{u \in \mathcal{U}(\R)} \bigg(\frac{2}{\kappa} \max_{g \in F} \| \pi(\Sigma_g) - u\rho(\Sigma_g) u^* \|_2  + \bigg(\sum_{i = 1}^\infty \frac{1}{2^i} \|\pi(Y_i) - u \rho(Y_i) u^* \|_2^2 \bigg)^{1/2}\bigg).$$
\end{defn}

This funny formula is computable in some cases (which is why we're using it).

\begin{prop}
\label{prop:inequality}
Let $\alpha_1, \alpha_2 \in \Aut(W^*(\Sigma(\Gamma), \{Y_i \}))$ be automorphisms that restrict to the identity on $\Sigma(\Gamma)$. Then $$\dist^* (\alpha_1 .[\iota], \alpha_2 .[\iota] ) = \bigg( \sum_i \frac{1}{2^i} \| \alpha_1(Y_i) - \alpha_2(Y_i) \|_2^2 \bigg)^{1/2}.$$  In particular, if $\alpha_1 \neq \id$, then $\alpha_1 .[\iota] \neq [\iota]$.
\end{prop}

\begin{proof} Since each $\alpha_i$ is the identity on $\Sigma(\Gamma)$, the inequality $\leq$ is trivial (let $u = 1$). For the other inequality, we fix a unitary $u \in (\IM_{k(n)}(\IC)\otimes R)^\omega$, define $\tilde{\e} :=\frac{2}{\kappa} \max_{g \in F} \| \iota(\Sigma_g) - u\iota(\Sigma_g) u^* \|_2$ and invoke Lemma \ref{lem:inequality} followed by Minkowski's inequality to observe that
\begin{align*}
\bigg( \sum_i \frac{1}{2^i} &\| \alpha_1(Y_i) - \alpha_2(Y_i) \|_2^2 \bigg)^{1/2}\\
& = \bigg( \sum_i (\frac{1}{2^{i/2}} \| \iota(\alpha_1(Y_i)) - u\iota(\alpha_2(Y_i))u^* +  u\iota(\alpha_2(Y_i))u^* - \iota(\alpha_2(Y_i)) \|_2)^2 \bigg)^{1/2}\\
&\leq \bigg( \sum_i (\frac{1}{2^{i/2}} \| \iota(\alpha_1(Y_i)) - u\iota(\alpha_2(Y_i))u^*\|_2 +  \frac{1}{2^{i/2}}   \tilde{\e})^2  \bigg)^{1/2}\\
&\leq  \bigg( \sum_i \frac{1}{2^{i}} \| \iota(\alpha_1(Y_i)) - u\iota(\alpha_2(Y_i))u^*\|_2^2\bigg)^{1/2} + \bigg( \sum_i \frac{1}{2^{i}} \tilde{\e}^2  \bigg)^{1/2}\\
&= \bigg( \sum_i \frac{1}{2^{i}} \| \iota(\alpha_1(Y_i)) - u\iota(\alpha_2(Y_i))u^*\|_2^2\bigg)^{1/2} + \frac{2}{\kappa}\max_{g \in F} \| \iota(\Sigma_g) - u\iota(\Sigma_g) u^* \|_2.
\end{align*}
Since $\dist^* (\alpha_1 .[\iota], \alpha_2 .[\iota] )$ is the infimum of the right hand side, the proof is complete.
\end{proof}

Of course, we are now left to wonder whether $W^*(\Sigma(\Gamma), \{Y_i \})$ has any nontrivial automorphisms that restrict to the identity on $\Sigma(\Gamma)$.  In general this is unclear because we have no idea what $W^*(\Sigma(\Gamma), \{Y_i \})$ looks like.  But fundamental work of Popa on freeness in ultraproducts implies that we can often identify $W^*(\Sigma(\Gamma), \{Y_i \})$ with a free product. More precisely, for every separable von Neumann algebra $M \subset \R$, there is a sequence of contractions $\{Y_i \} \subset (\IM_{k(n)}(\IC))^\omega$ such that $W^*(\{Y_i\}) \cong M$ and an isomorphism $W^*(\Sigma(\Gamma), \{Y_i \}) \cong \Sigma(\Gamma) \ast M$ (the tracial free product) that restricts to the identity on $\Sigma(\Gamma)$.

This fact is known to experts, so we only sketch the argument. First, with a judicious choice of microstates, one can find a trace-preserving embedding $M \subset (\IM_{k(n)}(\IC))^\omega$.\footnote{Though elementary, this is quite a technical exercise.  One first lifts the given embedding $M \subset \R$ to find microstates inside matrices $\IM_{l(m)}(\IC)$ that converge to (generators of) $M$ (in moments) as $m \to \infty$.  Then some careful bookkeeping, taking (not necessarily unital) direct sums of $\IM_{l(m)}(\IC)$ inside $\IM_{k(n)}(\IC)$ when $k(n) >> l(m)$, allows one to construct microstates inside  $\IM_{k(n)}(\IC)$ that converge to $M$ as $n \to \infty$.  And these microstates yield a trace-preserving embedding $M \subset (\IM_{k(n)}(\IC))^\omega$.}  Since both $\Sigma(\Gamma)$ and $M$ are separable, so is the von Neumann algebra they generate, and thus \cite{popa:asterisque} (see also \cite{V}) ensures the existence of a Haar unitary $u \in (\IM_{k(n)}(\IC))^\omega$ which is free from $W^*(\Sigma(\Gamma), M)$.  Hence (uniqueness of GNS representations implies) the von Neumann algebra generated by $\Sigma(\Gamma)$ and $uM u^*$ is isomorphic to $\Sigma(\Gamma) \ast M$.

Thus we have a large class of examples to which Proposition \ref{prop:inequality} applies.  Indeed, $\Aut(M)$ acts on $\Sigma(\Gamma) \ast M$ via free-product automorphisms $\id \ast \alpha$ and hence we've proved the following theorem.

\begin{thm}
\label{thm:fp} For every separable von Neumann algebra $M  \subset \R$,  $\Hom(\Sigma(\Gamma) \ast M, \R)$ has an extreme point $[\iota]$ (Lemma \ref{lem:commute}) with the property that its stabilizer under the action of $\Aut(M)$ (via $\alpha \mapsto \id \ast \alpha$) is trivial.

In fact, by Proposition \ref{prop:inequality}, if $M$ is generated by contractions $\{ Y_i \}$ which are used to define the metric $\dist^*$ (Definition \ref{defn:delta}), then for all $\alpha_1, \alpha_2 \in \Aut(M)$ we have $$\dist^* ((\id\ast \alpha_1) .[\iota], (\id \ast \alpha_2) .[\iota] ) = \bigg( \sum_i \frac{1}{2^i} \| \alpha_1(Y_i) - \alpha_2(Y_i) \|_2^2 \bigg)^{1/2}.$$
\end{thm}

Contrast the next result with Corollary \ref{cor:T}.

\begin{cor}
\label{cor:nondiscrete} If $M  \subset \R$ has a nontrivial trace-preserving sequence $\alpha_n \in \Aut(M)$ such that $\alpha_n \to \id_M$ in the point 2-norm topology (e.g., if $M$ is not abelian and atomic), then the extreme points of $\Hom(\Sigma(\Gamma) \ast M, \R)$ are not discrete.
\end{cor}

\begin{proof} Let $[\iota]$ be as in Theorem \ref{thm:fp}.  Then $(\id\ast \alpha_n) .[\iota]$ are also extreme points. Since $\alpha_n \neq \id$, $(\id\ast \alpha_n) .[\iota] \neq [\iota]$, but $\dist^* ((\id\ast \alpha_n) .[\iota], [\iota] ) \to 0$.
\end{proof}

If we had started with $\Gamma = SL(n,\Z)$ for some odd integer $n \geq 3$, then a striking theorem (cf.\ \cite[Theorem 1]{bekka}) of Bekka implies that $\Sigma(\Gamma) = L(\Gamma)$. (Many thanks to Sorin Popa for bringing Bekka's paper to our attention, and suggesting its relevance to this work.)  Hence, specializing to this case we have:

\begin{cor} Let $\Gamma = SL(n,\Z)$ for some odd integer $n \geq 3$ and $M \subset \R$ be any separable subalgebra.  Then $\Aut(M)$ acts on $\Hom(L(\Gamma)\ast M, \R)$ (via $\alpha \mapsto \id \ast \alpha$) and there is an extreme point $x \in \Hom(L(\Gamma)\ast M, \R)$ with trivial stabilizer (for the $\Aut(M)$ action). If $M$ is not abelian and atomic, then the extreme points of $\Hom(L(\Gamma)\ast M, \R)$ are not discrete.
\end{cor}

Specializing even further to the case $M = R$, we get ``proper" embeddings of arbitrary discrete groups.  That is, if $\Lambda$ is a countable discrete group, we can write $$R \cong \bigotimes_{s \in \Lambda} \IM_{2}(\IC)$$ and let $\Lambda$ act by Bernoulli shifts.  If $Y \in \IM_{2}(\IC)$ is a partial isometry with orthogonal range and support projections, then $Y$ generates $\IM_{2}(\IC)$ and so we can take $\{ Y_s \}_{s \in \Lambda}$ as our set of generators for $R$ and Theorem \ref{thm:fp} specializes to:

\begin{cor}
\label{cor:proper}
Let $\Gamma = SL(n,\Z)$ for an odd integer $n \geq 3$ and $\Lambda$ be any countable discrete group acting on $R$ by Bernoulli shifts.  Then taking free products with the identity map on $L(\Gamma)$, $\Lambda$ acts on $\Hom( L(\Gamma) \ast R, \R)$ and there is an extreme point $x_0 \in \Hom( L(\Gamma) \ast R, \R)$ such that $$\dist^*(s.x_0, t.x_0 ) = 2( \| Y \|_2^2 - |\tau(Y)|^2) = 1$$ for all distinct group elements $s, t \in \Lambda$.
\end{cor}

One could replace $R$ with $L(\mathbb{F}_\infty)$, where every $\Lambda$ acts  by free Bernoulli shifts, and get a similar result.

\section{Functorial issues and concluding remarks}
\label{sec:func}

\subsection{Rescalings}

There is a natural notion of isomorphism for the dynamical systems we've been considering.  Namely,  $(\Hom(N, \R), \Out(N))$ \emph{is isomorphic to} $(\Hom(M, \R), \Out(M))$ if there is an ``affine" homeomorphism $\Theta \colon \Hom(N, \R) \to \Hom(M, \R)$ and a group isomorphism $T \colon \Out(N) \to \Out(M)$ such that $\Theta( \alpha . x ) = T(\alpha). \Theta(x)$ for all $x \in \Hom(N, \R)$ and $\alpha \in \Out(N)$.

It turns out that the dynamical systems associated to $N$ and $pNp$ are isomorphic, for all projections $p \in N$, as we now prove.

\begin{subdefn}
\label{defn:comp}
Let $p \in N$ be a projection. Define $\Theta^p \colon \Hom(N, \R) \to \Hom(pNp, \R)$ by $\Theta^p([\pi]) = [\theta_{\pi(p)}\circ \pi|_{pNp}]$, where $\theta_{\pi(p)} \colon \pi(p) \R \pi(p) \to \R$ is a standard isomorphism.  Also, given $\alpha \in \Out(N)$, let $\alpha^p \in \Out(pNp)$ be the canonically associated outer automorphism.\footnote{Recall that if $\alpha \in \Aut(N)$ and $v_{\alpha} \in N$ is a partial isometry such that $v_{\alpha}^* v_{\alpha} = \alpha(p)$ and $v_{\alpha} v_{\alpha}^* = p$, then $\alpha^p (x) := v_{\alpha} \alpha(p) v_{\alpha}^*$ defines an automorphism of $pNp$, and this procedure $\alpha \mapsto \alpha^p$ descends to an isomorphism $\Out(N) \cong \Out(pNp)$ that is independent of all choices.}
\end{subdefn}

Our first lemma follows from the definition of standard isomorphism, hence will be left to the reader.

\begin{sublem}
\label{lem:suck} Given projections $s \leq t \in \R$ and standard isomorphisms $\theta_s \colon s\R s \to \R, \theta_t \colon t\R t \to \R$, the isomorphism $$\theta_s \circ \theta_t^{-1}|_{\theta_t(s) \R \theta_t(s)} \colon \theta_t(s) \R \theta_t(s) \to \R$$ is also standard.

Thus, $\theta_t \circ \theta_s^{-1} \colon \R \to \theta_t(s) \R \theta_t(s)$ is the inverse of a standard isomorphism.
\end{sublem}

\begin{sublem}
\label{lem:well} $\Theta^p \colon \Hom(N, \R) \to \Hom(pNp, \R)$ is well-defined, ``affine", continuous and covariant for the actions of $\Out(N)$ and $\Out(pNp)$.
\end{sublem}

\begin{proof}  Proving $\Theta^p$ is well defined is similar to arguments we've seen already (cf.\ Remark \ref{rem:idontknow}), hence will be left to the reader. Since the topology on $\Hom( \cdot, \R)$ is essentially point-2-norm convergence modulo unitary conjugation, it is routine to verify that $\Theta^p$ is continuous.

Checking covariance is only slightly harder.  Given $[\pi] \in \Hom(N, \R)$, a standard isomorphism $\theta_{\pi(p)} \colon {\pi(p)} \R {\pi(p)} \to \R$, $\alpha \in \Aut(N)$ and a unitary $u_{\alpha}$ such that $u_{\alpha} \alpha(p) u_{\alpha}^* = p$, the unital embedding $pNp \to \R$ given by $$x \mapsto \theta_{\pi(p)} \circ \pi |_{pNp} (u_{\alpha} \alpha(x) u_{\alpha}^*)$$ is a representative of $\alpha^p . \Theta^p([\pi])$.  A representative of $\Theta^p( \alpha . [\pi])$ is given by $$x \mapsto \theta_{\pi(\alpha(p))} \circ \pi \circ \alpha (x),$$ where $\theta_{\pi(\alpha(p))} \colon \pi(\alpha(p)) \R \pi(\alpha(p)) \to \R$ is any standard isomorphism.  For example, we could take $$\theta_{\pi(\alpha(p))} (y) := \theta_{\pi(p)} ( \pi(u_\alpha) y \pi( u^*_\alpha )),$$ which is easily seen to be a standard isomorphism, and then it is clear that the covariance condition $\Theta^p( \alpha . [\pi]) = \alpha^p . \Theta^p([\pi])$ is satisfied.

To see that $\Theta^p$ preserves the convex-like structure we fix $[\pi_1], \ldots, [\pi_k] \in \Hom(N, \R)$ and $t_1, \ldots, t_k \in [0,1]$ such that $\sum t_i = 1$.  Then fix some orthogonal projections $p_1, \ldots, p_k \in \R$ such that $\tau(p_i ) = t_i$ and standard isomorphisms $\theta_i \colon p_i \R p_i \to \R$.  Next, define $q_i = \theta_i^{-1} \circ \pi_i(p)$, $q = \sum_1^k q_i$ and fix standard isomorphisms $\theta_{q_i} \colon q_i \R q_i \to \R$ and $\theta_q \colon q\R q \to \R$.  Note that $\tau(q) =\tau(p)$ and $\tau(q_i) = t_i \tau(p) = t_i \tau(q)$; hence $\tau(\theta_q(q_i)) = t_i$.

Letting $s = q_i$ and $t = p_i$ in Lemma \ref{lem:suck}, we see that the isomorphism $$\theta_{q_i} \circ \theta_i^{-1}|_{\pi_i(p) \R \pi_i(p)} \colon  \pi_i(p) \R \pi_i(p) \to \R$$ is standard.  Similarly, but taking $s = q_i$ and $t = q$ this time, the isomorphism $\theta_q \circ \theta_{q_i}^{-1}\colon \R \to \theta_q(q_i) \R \theta_q(q_i)$ is the inverse of a standard isomorphism $\theta_q(q_i) \R \theta_q(q_i) \to \R$.  Since $\tau(\theta_q(q_i)) = t_i$, we can use the projections $\theta_q(q_i)$ and the standard isomorphisms $\theta_{q_i} \circ \theta_i^{-1}$ in the construction of $\sum t_i\Theta^p([\pi_i])$.  Also, we're at liberty to use the standard isomorphisms $\theta_{q_i} \circ \theta_i^{-1}|_{\pi_i(p) \R \pi_i(p)}$ in the construction of $\Theta^p [\pi_i]$.  Hence we have $$\sum t_i\Theta^p([\pi_i]) = \sum t_i [ ( \theta_{q_i} \circ \theta_i^{-1}) \circ \pi_i|_{pNp} ] = \bigg[  \sum (\theta_q \circ \theta_{q_i}^{-1} ) \circ ( \theta_{q_i} \circ \theta_i^{-1}) \circ \pi_i|_{pNp} \bigg].$$

On the other hand, we can use the $p_i$'s, $\theta_i$'s and $\theta_q$ in the construction of $\Theta^p \big( \sum t_i [\pi_i] \big)$. Thus we have
\begin{align*}
\Theta^p \big( \sum t_i [\pi_i] \big) & = \Theta^p \bigg( [ \sum \theta_i^{-1} \circ \pi_i ] \bigg)\\
& = \bigg[ \theta_q \circ \big( \sum \theta_i^{-1} \circ \pi_i\big)|_{pNp} \bigg]\\
&= \bigg[  \sum \theta_q \circ \theta_i^{-1} \circ \pi_i|_{pNp} \bigg]\\
&= \bigg[  \sum (\theta_q \circ \theta_{q_i}^{-1} ) \circ ( \theta_{q_i} \circ \theta_i^{-1}) \circ \pi_i|_{pNp} \bigg].\\
&= \sum t_i\Theta^p([\pi_i]),
\end{align*}
showing $\Theta^p$ is ``affine".
\end{proof}

The covariance of the previous lemma means that $\Theta^p$ induces a map at the level of dynamical systems, which we will continue to denote by $\Theta^p$.

\begin{subthm}
\label{thm:scale}
For every II$_1$-factor $N$ and nonzero projection $p \in N$, $$\Theta^p \colon  (\Hom(N,\R), \Out(N)) \to (\Hom(pNp,\R), \Out(pNp))$$ is an (affine) isomorphism.
\end{subthm}

\begin{proof} The only thing left to verify is that $\Theta^p$ is a homeomorphism, and this is easiest to do when $\tau(p) = 1/k$ for some $k \in \N$.  Indeed, in this case $N \cong   \IM_k(\mathbb{C}) \otimes pNp$ and an inverse for $\Theta^p$ is easy to describe.  Namely, if $\pi^p \colon pNp \to \R$ is given, then $\id_k \otimes \pi^p \colon \IM_k(\mathbb{C}) \otimes pNp \to (\IM_k(\mathbb{C}) \otimes R)^\omega$ defines an element of $\Hom(N, \R)$ and one checks that $\Theta^p( [ \id_k \otimes \pi^p ] ) = [\pi^p]$.  Since the assignment $\pi^p \mapsto \id_k \otimes \pi^p$ clearly induces a continuous map at the level of $\Hom$, this shows that $\Theta^p$ is a homeomorphism whenever $\tau(p) = 1/k$.

For the general case, first observe that if $p \leq q \in N$ then we can consider three compression maps $$\Theta^q \colon \Hom(N , \R) \to \Hom(qNq, \R), \ \Theta^p \colon \Hom(N , \R) \to \Hom(pNp, \R)$$ and $$\Theta^p_q \colon \Hom(qNq , \R) \to \Hom(pNp, \R).$$ These maps satisfy the relation $\Theta^p = \Theta^p_q \circ \Theta^q.$  Hence, if $q \in N$ is arbitrary and we take $p \leq q$ of trace $1/k$, for sufficiently large $k \in \N$, then it follows from the previous paragraph that $\Theta^q$ is injective. To see that $\Theta^q$ is surjective, it suffices to show that $\Theta^p_q$ is injective, since the relation above implies $\Theta^p_q$ is surjective.  But $\Theta^p_q$ must be injective because we can pick a projection $s \leq p$ such that $\tau(s) = \tau(q)/j$, for sufficiently large $j \in \N$, and repeat the argument above with $1, q$ and $p$ replaced by $q, p$ and $s$.
\end{proof}

For every $0 < t < \infty$, let $N^t$ denote the amplification of $N$ by $t$ (i.e., the corner of $N \bar{\otimes} B(\mathcal{H})$ determined by a projection of trace $t$).

\begin{subcor}
\label{cor:free}
For every $0 < t < \infty$, the dynamical systems $(\Hom(N,\R), \Out(N))$ and $(\Hom(N^t,\R), \Out(N^t))$ are (affinely) isomorphic.  In particular, for every $s, t \in (1, \infty)$, there is a canonical isomorphism $$(\Hom(L(\mathbb{F}_s),\R), \Out(L(\mathbb{F}_s))) \cong (\Hom(L(\mathbb{F}_t),\R), \Out(L(\mathbb{F}_t))),$$ where $L(\mathbb{F}_s)$ is the (interpolated) free group factor (cf.\ \cite{Ken}, \cite{Florin}).
\end{subcor}

\subsection{Opposite algebras}

There is a canonical isomorphism $$(\Hom(N,\R), \Out(N) ) \cong (\Hom(N^{op}, (\R)^{op}), \Out(N^{op})).$$ Since $R\cong R^{op}$, it follows that $\R \cong (\R)^{op}$ and  hence $\Hom(N, \R) \cong \Hom(N^{op}, \R)$. Thus our dynamical systems can't distinguish between an algebra and its opposite algebra.  However, it could be interesting to replace $\R$ with $M^{\omega}$, where $M$ isn't anti-isomorphic to itself, and see if the resulting dynamical systems associated to $N$ and $N^{op}$ are still isomorphic.

\subsection{Products on $\Hom(N,\R)$}

Andreas Thom pointed out that every $*$-homomorphism $\gamma \colon N \to N\bar{\otimes}N$ gives rise to a product on $\Hom(N,\R)$ as follows: $[\pi] \cdot_{\gamma} [\rho] := [ (\pi\otimes \rho) \circ \gamma]$.  (Where $\pi\otimes \rho \colon N\bar{\otimes}N \to \R\bar{\otimes}\R \subset \R$ is understood as in Remark \ref{rem:reformulate}.)  In particular, quantum groups such as $L(\Gamma)$ admit such homomorphisms and it might be interesting to study the resulting products.

\subsection{Other issues}

There are lots of basic questions that are presently unresolved, such as what happens when $N$ is an increasing union or decreasing intersection of factors, what happens with finite-index subfactors, and whether there are Kunneth-type theorems relating $\Hom(N_1 \bar{\otimes} N_2, \R)$ and/or $\Hom(N_1 \ast N_2, \R)$ with $\Hom(N_1, \R)$ and $\Hom(N_2, \R)$.

We also don't know if other W$^*$-properties, such as the Haagerup property (cf.\ \cite{me-n-taka}) or Ozawa's solidity (cf.\ \cite{O}), are reflected in $\Hom(N, \R)$ (as is the case for property (T) in Corollary \ref{cor:T}).

\subsection{Connes' Embedding Problem and Fixed Points}

Connes's Embedding Problem is equivalent to deciding whether or not $\Hom(N,\R)$ is non-empty for every separable II$_1$-factor $N$.  This tautology is certainly not helpful, but there is an angle worth pursuing.  Namely, it would be nice to resolve the following question.

\begin{subquestion}
\label{connes}
Given $N \subset \R$, fix a countable subgroup $\Gamma \subset \Out(N)$.  Does $\Hom(N, \R)$ have a $\Gamma$-fixed point?  In particular, if $N$ has property (T), does $\Hom(N,\R)$ have a $\Out(N)$-fixed point?\footnote{A theorem of Connes ensures $\Out(N)$ is countable in the property (T) case; see \cite[Theorem 12.1.19]{me-n-taka}.}
\end{subquestion}

A positive answer to Connes's Embedding Problem would imply a positive answer to this question, since it would provide an embedding $N \rtimes \Gamma \subset \R$, and the restriction of this embedding to $N$ would be a $\Gamma$-fixed point of $\Hom(N,\R)$.  Thus a counterexample to Question \ref{connes} would yield a counterexample to Connes's Embedding Problem. However, the convex-like structure of $\Hom(N,\R)$ might make it possible to construct fixed-points following something like Choquet theory. Of course, one would presumably need a Krein-Milman-type theorem first, to ensure a rich supply of extreme points.  

\subsection{Farah's suggestion}

Ilijas Farah suggested using the enormity of the fundamental group of $\R$ to define an $\mathbb{R}_+$-cone structure on $\Hom(N, \R)$, in hopes that the Grothendieck construction would then produce a vector space and a canonical embedding of $\Hom(N,\R)$.  We believe this is equivalent to considering the semigroup of unitary equivalence classes of (non-unital) $*$-homomorphisms $\pi \colon N \to B(H) \bar{\otimes} \R$ with the property that the trace of $\pi(1_N)$ is finite.  There are subtleties to worry about, but this idea looks very promising and we hope it will lead to new results in the near future.  This is work in progress.

\section{Appendix}

\begin{center}
By Narutaka Ozawa
\end{center}


I follow the notation, definitions and conventions of the main body of this paper (except using $\IM_m$ to denote finite-dimensional matrices).

\begin{thm}\label{thm:nonsepA}
If $N \subset \R$ is not hyperfinite, then $\Hom(N,\R)$
is not separable with respect to the metric $\dist$.
\end{thm}

\begin{lem}\label{lem:normlift}
Let $(M_n)$ be a sequence of finite von Neumann algebras and
$(M_n)^\omega$ be its tracial ultraproduct.
Let $a_1,\ldots,a_k \in (M_n)^\omega$ and $c_1,\ldots,c_k \in \IM_m$.
Then, there are $a_{i,n} \in M_n$ such that with
$a_i = (a_{i,n})_n \in (M_n)^\omega$, $\sup_n \| a_{i,n}\| = \| a_i \|$ and
\[
\sup_{n} \| \sum_{i=1}^k c_i\otimes a_{i,n} \|_{\IM_m\otimes M_n}
= \| \sum_{i=1}^k c_i\otimes a_i \|_{\IM_m\otimes (M_n)^\omega}.
\]
Moreover, if $a_1=1$, then we can take $a_{1,n} \in \IC1$ for all $n$.
\end{lem}
\begin{proof}
Recall that $(M_n)^\omega=(\prod M_n)/K_\tau$, where
\[
K_{\tau} = \{ (x_n) \in \prod_{n=1}^\infty M_n : \lim_{n\to\omega} \|x_n\|_2 = 0 \}.
\]
Let $(b_{i,n})_n \in \prod M_n$ be a norm-preserving lift of $a_i$.
In case where $a_1=1$, take $b_{1,n}=1$ for all $n$.
Since the element $\sum c_i\otimes a_i$ in $\IM_m\otimes (M_n)^\omega$ has a
norm-preserving lift in $\IM_m\otimes\prod M_n$,
there is $(z_n) \in \IM_m\otimes K_{\tau}$ such that
\[
\sup_n \| z_n + \sum_{i=1}^k c_i\otimes b_{i,n} \|_{\IM_m\otimes M_n}
= \| \sum_{i=1}^k c_i\otimes a_i \|_{\IM_m\otimes (M_n)^\omega}.
\]
Since $\IM_m$ is finite-dimensional,
one can find projections $p_n\in M_n$ such that
\[
\lim_{n\to\omega} \tau(p_n) = 0
\mbox{ and }
\lim_{n\to\omega}\|z_n(1\otimes p_n^\perp)\|=0.
\]
This follows from the facts that for an element $x$ in a finite von Neumann algebra,
the spectral projection $p=\chi_{[\e,\infty)}(|x|)$ satisfies
$\tau(p)\le\tau(|x|)/\e$ and $\| xp^\perp \|\le\e$;
and that for any projections $\{p_j\}$, one has $\tau(\bigvee p_j)\le\sum_j\tau(p_j)$.
Thus
\[
b_{i,n}' := p_n^\perp b_{i,n} p_n^\perp + \tau(a_i)p_n
\]
satisfy that
$a_i = (b_{i,n}')_n$ in $M$ and
\[
\lim_{n\to\omega} \| \sum_{i=1}^k c_i\otimes b_{i,n}' \|_{\IM_m\otimes M_n}
= \| \sum_{i=1}^k c_i\otimes a_i \|_{\IM_m\otimes (M_n)^\omega}.
\]
Consequently,
\[
a_{i,n} := \min\{ \frac{\| \sum_{j=1}^k c_j\otimes a_j \|_{\IM_m\otimes (M_n)^\omega}}
{\| \sum_{j=1}^k c_j\otimes b_{j,n}' \|_{\IM_m\otimes M_n}},1\}\, b_{i,n}'
\]
satisfy the desired conditions.
\end{proof}

\begin{proof}[Proof of Theorem \ref{thm:nonsepA}]
Let $N \subset \R$ be non-hyperfinite. Then $N$ is also embeddable into the ultraproduct
of matrix algebras $(\IM_n)$, which is denoted by $(\IM_n)^\omega$. We regard $N\subset (\IM_n)^\omega$.
By Lemma 2.2 in \cite{haagerup:dec}, there are a non-zero central projection $p \in N$
and a finite tuple of unitary elements $u_1,\ldots,u_k \in Np$ such that
\[
\delta:=\| \frac{1}{k}\sum_{i=1}^k u_i\otimes\bar{u}_i \|_{N\otimes\bar{N}}<1.
\]
It suffices to show $\Hom(Np,R^\omega)$ is non-separable.
So, we assume $p=1$.
Let $C$ be the universal C$^*$-algebra generated by a sequence $(X_i)_{i=1}^\infty$
of contractions and fix a $*$-homomorphism $\theta\colon C\to N$ such that
$\theta(X_i)=u_i$ for $i=1,\ldots,k$ and that $\theta(C)$ is weakly dense in $N$.
We also fix a metric $d$ on the state space of $C$ which induces the weak$^*$-topology.
We will inductively find an increasing sequence $l(n) \in \IN$
and $*$-homomorphisms $\theta_n\colon C \to \IM_{l(n)}$ such that
$\tau_{l(n)}\circ\theta_n\to\tau\circ\theta$ and
\[
\sup_{m\neq n}\| \frac{1}{k}\sum_{i=1}^k \theta_n(X_i)\otimes\bar{\theta}_m(\bar{X}_i)
 \|_{\IM_{l(n)}\otimes\bar{\IM}_{l(m)}}\le\delta^{1/2}.
\]
Indeed, suppose $\theta_m$'s are given for $m=1,\ldots,n-1$, and let 
$c_i = \bigoplus_{m=1}^{n-1} \bar{\theta}_m(\bar{X}_i)$ for $i=1,\ldots,k$.
By Haagerup's Cauchy--Schwarz inequality (Lemma~2.4 in \cite{haagerup:dec}),
one has
\[
\| {1 \over k}\sum_{i=1}^k u_i \otimes c_i \| \le
 \| {1 \over k}\sum_{i=1}^k u_i \otimes \bar{u}_i \|^{1/2}
 \| {1 \over k}\sum_{i=1}^k \bar{c}_i \otimes c_i \|^{1/2}
 \le \delta^{1/2}.
\]
Thus, Lemma~\ref{lem:normlift}, applied to $a_i = \theta(X_i)$ for $i=1,\ldots,k(n)$
(where $k(n)$ is large enough) and $c_i$ (let $c_i=0$ for $i>k$), implies that there are $l(n)$ and 
contractions $x_{i,n} \in \IM_{l(n)}$ for $i=1,\ldots,k(n)$
(and let $x_{i,n}=0$ for $i>k(n)$)
such that the $*$-homomorphism $\theta_n$ defined by $\theta_n(X_i)=x_{i,n}$ satisfies
$d(\tau_{l(n)}\circ\theta_n,\tau\circ\theta)<1/n$ and
\[
\max_{m=1,\ldots,n-1}\| \frac{1}{k}\sum_{i=1}^k \theta_n(X_i)\otimes\bar{\theta}_m(\bar{X}_i)
 \|_{\IM_{l(n)}\otimes\bar{\IM}_{l(m)}}\le\delta^{1/2}.
\]
We embed each $\IM_{k(n)}$ into $R$ and regard $\theta_n$ as $*$-homomorphisms into $R$.
For a subsequence $\alpha\colon \IN \to \IN$, we define
$\theta_\alpha\colon C\to R^\omega$ by $\theta_\alpha(x) = (\theta_{\alpha(n)})_n$.
Since $\tau\circ\theta_n\to\tau\circ\theta$, the von Neumann algebra generated by
$\theta_\alpha(C)$ is canonically isomorphic to $N$ and
there is $\pi_\alpha \in \Hom(N,R^\omega)$ such that
$\pi_\alpha\circ\theta=\theta_\alpha$.
Let $\beta$ be another subsequence such that $\alpha(n)\neq\beta(n)$
for all but finitely many $n$.
Then, we claim that
\[
\sup_{v \in \cU(R^\omega)}\|\frac{1}{k}\sum_{i=1}^k \pi_\alpha(u_i) \, v \, \pi_\beta(u_i)^*\|_2\le\delta^{1/2}.
\]
This concludes the non-separability of $\Hom(N,R^\omega)$.
To prove the claim, let $(v_n)_n$ be a sequence of unitary elements in $R$
such that $(v_n)_n = v$ in $R^\omega$.
Since $R\otimes\bar{R}$ acts on $L^2(R)$ by $(a\otimes\bar{b})\xi=a\xi b^*$, one has
\begin{align*}
\|\frac{1}{k}\sum_{i=1}^k \pi_\alpha(u_i) \, v \, \pi_\beta(u_i)^*\|_2
 &= \lim_{n\to\omega} \|\frac{1}{k}\sum_{i=1}^k \theta_{\alpha(n)}(X_i)
  \, v_n \, \theta_{\beta(n)}(X_i)^*\|_2\\
 &\le \lim_{n\to\omega} \|\frac{1}{k}\sum_{i=1}^k \theta_{\alpha(n)}(X_i)
  \otimes \bar{\theta}_{\beta(n)}(\bar{X}_i)\|_{2,R\otimes\bar{R}}\\
 &\le \delta^{1/2}.
\end{align*}
\end{proof}

Lemma~\ref{lem:normlift} has another interesting consequence.
\begin{prop}
Let $(M_n)$ be a sequence of finite von Neumann algebras and
$Q\colon\tilde{M}\to M$ be the canonical quotient map from
the norm-ultraproduct $\tilde{M}$ onto the tracial ultraproduct $M$.
Then, for any separable C$^*$-algebra $A\subset M$, there exists
a contractive completely positive map $\p\colon A\to\tilde{M}$ such that
$Q\circ\p=\id_A$.
\end{prop}
\begin{proof}
We may assume that $A$ contains the unit of $M$.
Let $C$ be the universal C$^*$-algebra generated by a sequence $(X_i)_{i=1}^\infty$
of contractions and fix a surjective $*$-homomorphism $\theta\colon C\to A$
such that $\theta(X_1)=1$.
Let $a_i=\theta(X_i)$.
Let $\IM_n\hookrightarrow\IK$ be the standard embedding into the left upper corner
and $\IK_0=\bigcup\IM_n$.
Choose $c^{(l)}_i \in \IK_0$ so that for every $k$ the sequence
$(c^{(l)}_1,\ldots,c^{(l)}_k)_{l=1}^\infty$ is dense in $\IK^k$.
By Lemma~\ref{lem:normlift}, there are contractions $a^{(l)}_{i,n} \in M_n$ such that
$(a^{(l)}_{i,n})_n=a_i$ in $M$ for every $i$ and $l$, and that
\[
\sup_{n} \| \sum_{i=1}^k c^{(m)}_i\otimes a^{(l)}_{i,n} \|_{\IK\otimes M_n}
= \| \sum_{i=1}^k c^{(m)}_i\otimes a_i \|_{\IK\otimes M}
\]
for every $k,m$ and $l$ with $l>\max\{k,m\}$.
Let $\theta^{(l)}_n\colon C\to M_n$ be the $*$-homomorphism defined by
$\theta^{(l)}_n(X_i)=a^{(l)}_{i,n}$.
Let $d$ be a metric on the state space of $C$ which induces the weak$^*$-topology.
Then, for each $l$, one has $\lim_{n\to\omega}d(\tau\circ\theta^{(l)}_n,\tau\circ\theta)=0$.
For each $n$, we define $l(n)\in\IN$ by
\[
l(n) = \max\{ l=1,\ldots,n : d(\tau\circ\theta^{(l)}_n,\tau\circ\theta) < \frac{1}{l-1} \}.
\]
Since $l(n)\geq l$ on $\{ n : d(\tau\circ\theta^{(l)}_n,\tau\circ\theta) < 1/(l-1) \}$
and the latter set belongs to $\omega$, one has $\lim_{n\to\omega} l(n)=+\infty$.
Let $\theta_n=\theta^{(l(n))}_n$ and $a_{i,n}=\theta_n(X_i)$.
Then, one has
$\lim_{n\to\omega}d(\tau\circ\theta_n,\tau\circ\theta)
\le\lim_{n\to\omega}1/(l(n)-1)=0$, and
\[
\lim_{n\to\omega} \| \sum_{i=1}^k c^{(m)}_i\otimes a_{i,n} \|_{\IK\otimes M_n}
\le \| \sum_{i=1}^k c^{(m)}_i\otimes a_i \|_{\IK\otimes M}
\]
for all $k,m\in\IN$.
Therefore we can define the unital and completely contractive (and hence completely positive)
map $\p\colon A\to\tilde{M}$ by $\p(a_i)=(a_{i,n})_n$ in $\tilde{M}$.
Since $\tau\circ\theta_n\to_\omega\tau\circ\theta$, one has $Q\circ\p=\id_A$.
\end{proof}

I don't know whether one can take $\p$ to be a $*$-homomorphism.
If this is true, its proof would be very difficult in light of
Haagerup and Thorbj{\o}rnsen's results (cf.\ \cite{HT}).

\begin{lem}
Let $\Lambda$ be a finite group generated by a symmetric subset $S\subset\Lambda$
and
\[
\delta:=\|\frac{1}{|S|}\sum_{g\in S}\lambda^0(g) \|<1,
\]
where $\lambda^0$ is the restriction of the regular representation
$\lambda$ to $\ell_2\Lambda\ominus\IC{\mathbf 1}$.
Suppose that $(\alpha_g)_{g\in S}$ are complex numbers of modulus $1$ and
$\xi\in\ell_2\Lambda$ is a unit vector such that
\[
\e:=\max_{g\in S}\| \lambda(g)\xi - \alpha_g\xi \|>0.
\]
Then there is a character\footnote{By a character, we mean a unitary character.}
$\chi$ on $\Lambda$ such that
\[
\max_{g\in S}| \bar{\chi}(g) - \alpha_g | \le \left(\frac{100\e}{1-\delta}\right)^{1/2}.
\]
\end{lem}
\begin{proof}
We may assume that $100\e/(1-\delta)<4$.
We consider the unitary representation $\sigma=\Ad\lambda$ of $\Lambda$
on $\IB(\ell_2\Lambda)$ equipped with the Hilbert--Schmidt norm,
and note that $\sigma$ (or any other unitary representation) is contained in a multiple of $\lambda$.
Then, for the rank one operator $X:=\xi\otimes\bar{\xi}$, one has
$\| X - \sigma(g)X\|_{2,\Tr} \le 2\e$ for every $g\in S$.
By assumption, the contraction $h=\frac{1}{|S|}\sum_{g\in S}\sigma(g)$ has spectrum
contained in $[-1,\delta]\cup\{1\}$ and the spectral projection corresponding to
the spectral subset $\{1\}$ is the projection $\frac{1}{|\Lambda|}\sum_{g\in\Lambda}\sigma(g)$
onto $\sigma$-invariant vectors.
Since $\|X-hX\|_{2,\Tr}\le 2\e$, the positive operator
$Y=\frac{1}{|\Lambda|}\sum_{g\in\Lambda}\sigma(g)X$
---write $Y=\sum_i \mu_i\zeta_i\otimes\bar{\zeta}_i$ for spectral decomposition---
satisfies
\[
1 - \left(\frac{2\e}{1-\delta}\right)^2 \le \Tr(Y^*Y) = \Tr(X^*Y)
= \sum_i \mu_i|\ip{\zeta_i,\xi}|^2.
\]
Let's write $\e_0=2\e/(1-\delta)$ for simplicity.
Since $\sum_i|\ip{\zeta_i,\xi}|^2\le1$, there is $i_0$ such that $\mu_{i_0}\geq 1- \e_0^2$.
Moreover, $\mu_i\le (1-\mu_{i_0}^2)^{1/2}\le 2^{1/2}\e_0$ for all $i\neq i_0$.
It follows that
\[
|\ip{\zeta_{i_0},\xi}|^2 \geq 1 - \e_0^2 -\sum_{i\neq i_0}\mu_i|\ip{\zeta_i,\xi}|^2 \geq 1-2\e_0.
\]
Hence, for $\zeta=\gamma\zeta_{i_0}$, where $\gamma$ is a complex number of modulus $1$
such that $\ip{\zeta,\xi}>0$,
one has $\| \zeta - \xi \|^2\le2(1-\ip{\zeta,\xi}^2)\le 4\e_0$ and
the rank one operator $\zeta\otimes\bar{\zeta}$
is a spectral projection of $Y$ and thus $\sigma$-invariant.
This means that $|\zeta|^2$ is constant (whose value is $|\Lambda|^{-1}$)
on $\Lambda$, and $\chi(g)=|\Lambda|\zeta(g)\bar{\zeta}(1)$
is a character on $\Lambda$.
Moreover,
\begin{align*}
|\bar{\chi}(g)-\alpha_g|^2 &= \|\lambda(g)\zeta - \alpha_g\zeta \|^2\\
&\le 2\|\lambda(g)\xi - \alpha_g\xi \|^2+ 2\|(\lambda(g)-\alpha_g)(\zeta-\xi)\|^2\\
&\le 2\e^2+32\e_0
\end{align*}
for every $g\in S$.
\end{proof}

Let $\G$ be a residually finite group and $\pi_n\colon\G\to\G_n$ be
a sequence of finite quotients such that for every $g\neq1$
one has $\pi_n(g)\neq1$ eventually as $n\to\infty$.
Then, the sequence $\pi_n\colon\G\to L(\G_n)\hookrightarrow R$ gives
rise to an embedding $\pi$ of $L(\G)$ into $R^\omega$.
The class $[\pi]\in\Hom(L(\G),R^\omega)$ does not depend on
the choice of $L(\G_n)\hookrightarrow R$.
Recall from Definition 4.3.1 in \cite{lubotzky} that
a finitely generated group $\G$, say generated by a finite symmetric subset $S$,
is said to have property $(\tau)$ with respect to the family $\{ \pi_n\}$ if
\[
\sup_n\|\frac{1}{|S|}\sum_{g\in S}\lambda_{\G_n}^0(g)\|<1,
\]
where $\lambda_{\G_n}^0$ is the obvious unitary representation
of $\G$ on $\ell_2(\G_n)\ominus\IC\mathbf{1}$.

Let $\chi$ be a character on $\G$. Then, the map
$\lambda(g)\mapsto\chi(g)\lambda(g)$ extends to a $*$-automorphism on $L(\G)$,
which is still written as $\chi$.

\begin{thm}\label{thm:charaction}
Let $\G$ and $\pi$ be as above,
and assume that $\G$ has property $(\tau)$ with respect to $\{ \pi_n\}$, and that
$\G_n$ have no non-trivial characters.
Then, for every non-trivial character $\chi$ on $\G$,
one has $[\pi]\neq[\pi\circ\chi^{-1}]$ in $\Hom(L(\G),R^\omega)$.
\end{thm}

\begin{proof}
Suppose that $[\pi]=[\pi\circ\chi^{-1}]$.
Then, there is a sequence $(u_n)$ of unitary elements in $R$ such that
$\|u_n^*\pi_n(g)u_n-\bar{\chi}(g)\pi_n(g)\|_2\to0$ for every $g\in\G$.
It follows that for the representation $\sigma_n:=\Ad\pi_n$ of $\G$ on $L^2(R)$,
the vector $u_n \in L^2(R)$ is almost invariant under $\chi(g)\sigma_n(g)$.
Since $\sigma_n$ factors through $\pi_n$, it is contained in a multiple of
$\lambda_{\G_n}\colon\G\curvearrowright\ell_2(\G_n)$.
Hence, the representation $\chi(g)\lambda_{\G_n}(g)$ has an almost invariant vector.
By the previous lemma, there are characters $\chi_n$ on $\G_n$ such that
$\chi_n\circ\pi_n\to\chi$ pointwise.
But since $\G_n$ has no non-trivial characters, all $\chi_n$ are trivial and so is $\chi$.
\end{proof}

\begin{cor}
Every non-trivial character on $\IF_r$ acts non-trivially on $\Hom(L(\IF_r),R^\omega)$.
\end{cor}
\begin{proof}
For simplicity, we assume $r=2$.
It is sufficient to exhibit a sequence $(\pi_n)$ satisfying the assumption
of Theorem \ref{thm:charaction}.
It is well-known that the subgroup of $\mathrm{PSL}(2,\IZ)$, generated by
$\left[\begin{smallmatrix} 1 & 2 \\ 0 & 1\end{smallmatrix}\right]$ and
$\left[\begin{smallmatrix} 1 & 0 \\ 2 & 1\end{smallmatrix}\right]$,
is isomorphic to $\IF_2$ and has finite index.
Take an increasing sequence $(p_n)$ of prime numbers larger than $3$ and consider
\[
\pi_n\colon \IF_2\hookrightarrow\mathrm{PSL}(2,\IZ)\to\mathrm{PSL}(2,\IZ/p_n\IZ)=\G_n.
\]
Note that all $\pi_n$ are surjective.
Property $(\tau)$ follows from Selberg's theorem (see Example 4.3.3.D in \cite{lubotzky})
plus the fact that $\IF_2$ has finite index in $\mathrm{PSL}(2,\IZ)$.
Moreover, $\mathrm{PSL}(2,\IZ/p_n\IZ)$ are simple and have no non-trivial characters.
\end{proof}


\end{document}